\documentclass[a4paper,12pt]{article}
\usepackage{amsthm}
\usepackage{amsmath,amssymb,latexsym,amsfonts,mathrsfs}
\usepackage{bbm}
\usepackage[pdftex]{graphicx}
\usepackage{multirow}
\usepackage{hhline}

\newcommand{\dist}{\stackrel{{\rm d}}{=}}
\newcommand{\dto}{\stackrel{{\rm d}}{\longrightarrow}}

\renewcommand{\bar}{\overline}

\newcommand{\bC}{\ensuremath{\mathbb{C}}}

\newcommand{\bE}{\ensuremath{\mathbb{E}}}

\newcommand{\bN}{\ensuremath{\mathbb{N}}}

\newcommand{\bP}{\ensuremath{\mathbb{P}}}

\newcommand{\bR}{\ensuremath{\mathbb{R}}}
\newcommand{\bS}{\ensuremath{\mathbb{S}}}

\newcommand{\cB}{\ensuremath{\mathcal{B}}}

\newcommand{\cF}{\ensuremath{\mathcal{F}}}

\newcommand{\cL}{\ensuremath{\mathcal{L}}}

\newcommand{\cP}{\ensuremath{\mathcal{P}}}

\newcommand{\cR}{\ensuremath{\mathcal{R}}}

\theoremstyle{plain}
\newtheorem{Thm}{Theorem}[section]

\newtheorem{Lem}[Thm]{Lemma}
\newtheorem{Prop}[Thm]{Proposition}
\newtheorem{Cor}[Thm]{Corollary}

\theoremstyle{definition}
\newtheorem{Ass}[Thm]{Assumption}
\newtheorem{Def}[Thm]{Definition}
\newtheorem{Rem}[Thm]{Remark}
\newtheorem{Ex}[Thm]{Example}

\numberwithin{equation}{section}

\newcommand{\bn}{\mbox{\boldmath $n$}}
\newcommand{\rd}{\ensuremath{\mathrm{d}}}

\begin{document}

\begin{center}
{\Large \bf 
Functional limit theorem for occupation time processes of intermittent maps
}
\end{center}
\begin{center}
Toru Sera
\end{center}

\begin{abstract}
We establish a functional limit theorem for the joint-law of 
occupations near and away from indifferent fixed points of interval maps, and of waits for the occupations away from these points, in the sense of strong distributional convergence.
It is a functional and joint-distributional extension of Darling--Kac type limit theorem, of Lamperti type generalized arcsine laws for occupation times, and of Dynkin and Lamperti type generalized arcsine laws for waiting times, at the same time. 
\end{abstract}



\section{Introduction} 
In statistical physics (see for example Pomeau--Manneville \cite{PM} and Manneville \cite{M}), interval maps with indifferent fixed points have been studied as models of intermittent phenomena, such as intermittent transitions to turbulence in convective fluid. In this context, the occupations near the indifferent fixed points correspond to \emph{laminar} phases, while the occupations away from them correspond to \emph{turbulent} bursts.

In the present paper, we study the time evolution of the occupations near and away from these points, and of the waits for the occupations away from these points. We obtain the scaling limit of the time evolution in the sense of strong distributional convergence.
Our limit theorem is a functional and joint-distributional extension  of Darling--Kac type limit theorems \cite{A81, A86, TZ, Z07, KZ, OSb}, of Lamperti type generalized arcsine laws for occupation times \cite{T02, TZ, Z07, SY}, and of Dynkin and Lamperti type generalized arcsine laws for waiting times \cite{T98, TZ, Z07, KZ}, at the same time.

Let us illustrate earlier studies.
Following \cite[Examples]{T02}, we define an interval map $T:[0,1]\to[0,1]$ by 
\begin{align}\label{boole}
	Tx
	=
	\frac{x(1-x)}{1-x-x^2}
	\quad
	\bigg(0\leq x \leq \frac{1}{2}\bigg),
	\quad
	Tx=1-T(1-x)
	\quad
	\bigg(\frac{1}{2}<x\leq 1\bigg).
\end{align} 
See Figure \ref{fig:graph}. Note that $0$ and $1$ are  \emph{indifferent fixed points} of $T$, that is, 
 $T0=0$, $T1=1$ and $T'0=T'1=1$.
In addition, we have $Tx= x+x^3+o(x^3)$, as $x\to 0$. This map is conjugate to \emph{Boole's transformation} \cite{AW}, which is a typical example of infinite ergodic transformations.
The map $T$ has an ergodic invariant measure 
\begin{align}
	\mu(\rd x)=(x^{-2}+(1-x)^{-2})\rd x, \quad 0<x<1,
\end{align}   
whose total mass is infinite. Let us fix $\delta\in(0,1/2)$ from now on.
With respect to $\mu$, the interval $[\delta,1-\delta]$ has finite mass and its complement $[0,\delta)\cup(1-\delta,1]$ has infinite mass. 
Birkhoff's pointwise ergodic theorem implies
\begin{align}
	\frac{1}{n}\sum_{k=1}^{n}\mathbbm{1}_{\{T^k x\in [\delta,1-\delta]\}}
	\underset{n\to\infty}{\to} 0
	\quad
	\bigg(\text{or equivalently,}\;\; \frac{1}{n}\sum_{k=1}^{n}\mathbbm{1}_{\{T^k x\notin [\delta,1-\delta]\}}
	\underset{n\to\infty}{\to} 1 \bigg), \quad
	\text{a.e.$x$.}
\end{align}
Roughly speaking, the orbit $(x,Tx,T^2x,\dots)$ of almost every initial point $x$ is concentrated close to $0$ and $1$. See Figure \ref{fig:orbit}.
\begin{figure}
\begin{minipage}{0.3\hsize}
\centering
	\includegraphics[height=4cm]{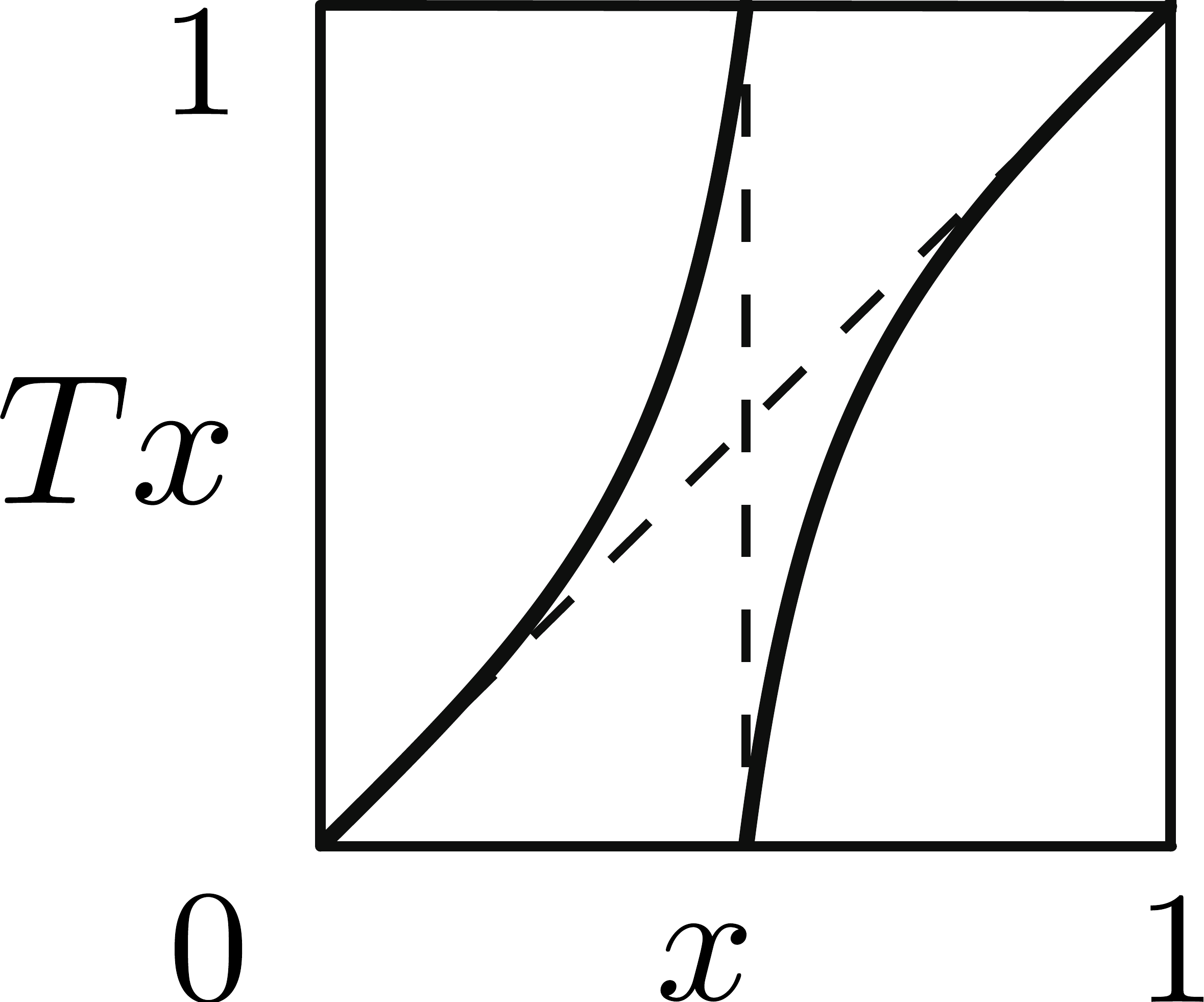}
	\caption{Graph of $T$}
	\label{fig:graph}
\end{minipage}
\begin{minipage}{0.7\hsize}
\centering
	\includegraphics[height=4cm]{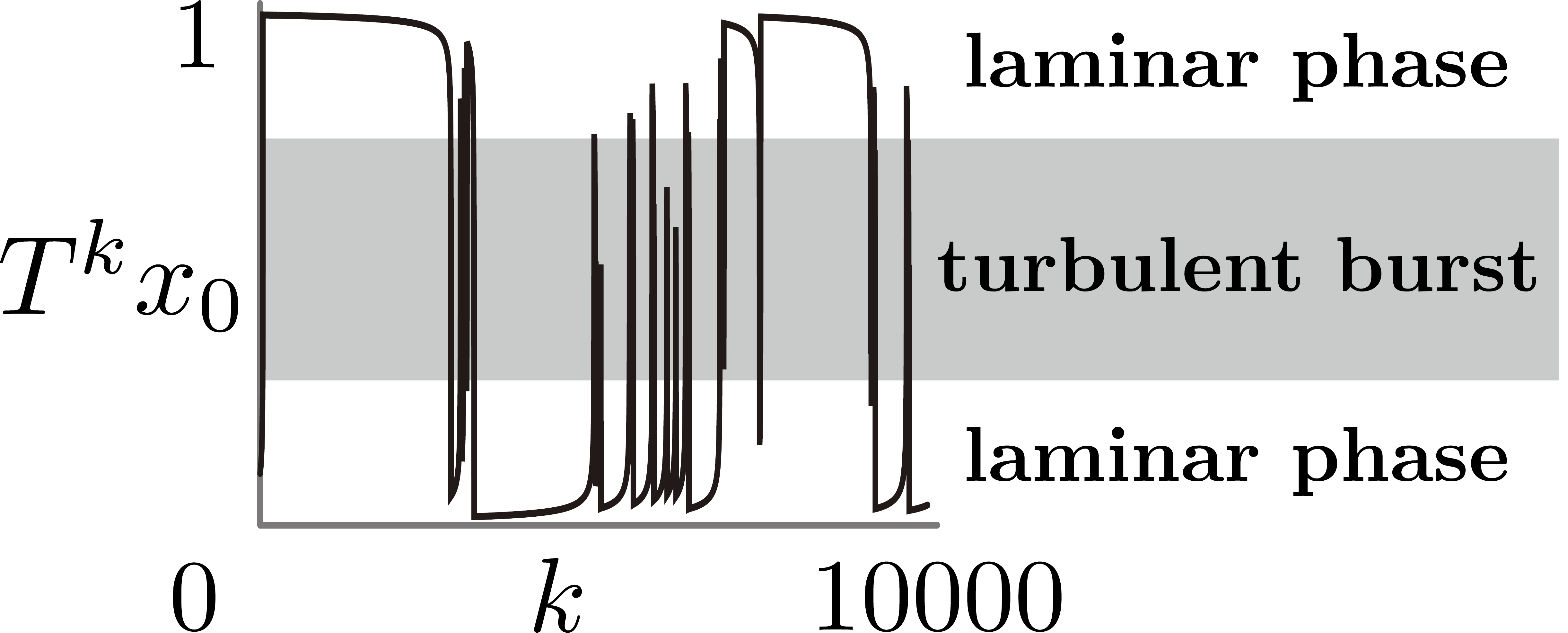}
	\caption{The orbit $(x_0, Tx_0,T^2 x_0,\dots)$ starting at $x_0=1/10$}
	\label{fig:orbit}	
\end{minipage}
\end{figure}
We are interested in non-trivial scaling limits of the occupation times for $[0,\delta)$, $[\delta,1-\delta]$ and $(1-\delta,1]$. Unfortunately, we cannot obtain them in the sense of a.e.$x$ convergence. See \cite{A77, ATZ} for the details. Instead, we can  obtain them in the sense of \emph{strong distributional convergence}, as we shall see in the following.

On the one hand, Aaronson \cite[Theorem 1]{A81} obtained  \emph{Darling--Kac type limit theorem} for the occupations away from the indifferent fixed points: for $t>0$ and $C=\sqrt{2\pi}/\mu([\delta,1-\delta])$,
\begin{align}\label{Aaronson}
	\frac{C}{\sqrt{n}}\sum_{k=1}^{\lfloor nt\rfloor}\mathbbm{1}_{\{T^k x\in [\delta, 1-\delta]\}}
	\;\underset{n\to\infty}{\overset{\cL(\mu)}{\Longrightarrow}}\;
	L(t)=\lim_{\varepsilon\downarrow0}\frac{1}{\sqrt{2}\varepsilon}\int_0^t \mathbbm{1}_{\{|B(s)|\leq \varepsilon\}}\rd s,
	\quad \text{in $\bR$},
\end{align}
where the notation $\overset{\cL(\mu)}{\Longrightarrow}$ denotes the strong distributional convergence with respect to $\mu$ (it will be explained in Section \ref{sec:notation}), and $B=(B(t):t\geq0)$ denotes a one-dimensional Brownian motion started at the origin, defined on a probability space $(\Omega, \cF, \bP)$. In other words, for any absolutely continuous probability measure $\nu$ on $[0,1]$, it holds that
\begin{align*}
  \nu\bigg[x\in [0,1]:\frac{C}{\sqrt{n}}\sum_{k=1}^{\lfloor nt\rfloor}\mathbbm{1}_{\{T^k x\in [\delta, 1-\delta]\}} \leq u\bigg]
  \underset{n\to\infty}{\to}
	\bP\big[L(t)\leq u\big]
	=
	\int_0^{u} \frac{e^{-s^2/(4t)}}{\sqrt{\pi t}}\rd s,
	\quad
	u\geq0.	
\end{align*}
The limit random variable $L(t)$ is called the \emph{Brownian local time} in the Blumenthal--Getoor normalization. Its one-dimensional marginal law is a half Gaussian distribution, that is, a Mittag-Leffler distribution of order $1/2$.
In addition, Aaronson \cite[Theorem 5']{A86} and Owada--Samorodnitsky \cite[Theorem 6.1]{OSb} obtained  a functional extension:
\begin{align}
   \bigg(
   \frac{C}{\sqrt{n}}\sum_{k=1}^{\lfloor nt\rfloor}\mathbbm{1}_{\{T^k x\in [\delta, 1-\delta]\}}:t\geq0\bigg)
   \;\underset{n\to\infty}{\overset{\cL(\mu)}{\Longrightarrow}}\;
   \big( L(t)  :t\geq0\big),	
   \quad \text{in $D([0,\infty),\bR)$}, \label{Ow-Sa}
\end{align}
where $D([0,\infty),\bR)$ denotes the space of c\`adl\`ag functions $w:[0,\infty)\to\bR$, endowed with the Skorokhod $J_1$-topology. The convergence (\ref{Ow-Sa}) is stronger than (\ref{Aaronson}), since distributional convergence in $D([0,\infty),\bR)$ implies finite-dimensional convergence.  

  On the other hand, Thaler \cite[Theorem]{T02} obtained  \emph{Lamperti type generalized arcsine laws} for the occupations near the indifferent fixed points: for any $t>0$,
\begin{align}\label{Thaler}
   \frac{1}{n}\sum_{k=1}^{\lfloor nt\rfloor}\mathbbm{1}_{\{T^k x <  \delta\}}   
   \;\underset{n\to\infty}{\overset{\cL(\mu)}{\Longrightarrow}}\;
   Z_-(t)
   =\int_0^t \mathbbm{1}_{\{B(s)<0\}}\rd s,
   \quad
   \text{in $\bR$}.	
\end{align}
The limit random variable $Z_-(t)$ denotes the amount of time which $B$ spends on the negative side $(-\infty,0)$ up to time $t$. Its  one-dimensional marginal law is an arcsine distribution. The convergence (\ref{Thaler}) is equivalent to the following: for any absolutely continuous probability measure $\nu$ on $[0,1]$, it holds that
\begin{align*}
 \nu\bigg[x\in[0,1]:
 \frac{1}{n}\sum_{k=1}^{\lfloor nt\rfloor}\mathbbm{1}_{\{T^k x < \delta\}}\leq u
 \bigg]
 \underset{n\to\infty}{\to}
 \bP\big[Z_-(t)\leq u \big]
 =\frac{2}{\pi}\text{arcsin}\bigg(\sqrt{\frac{u}{t}}\bigg),
 \quad 0\leq u\leq t.	
\end{align*}

 We now illustrate our main result. In Theorem \ref{main2} (see also Example \ref{ex:boole}), we obtain a functional convergence of the occupations near and away from the indifferent fixed points: 
 \begin{align}
		&\bigg( \frac{1}{n}\sum_{k=1}^{\lfloor nt\rfloor}\mathbbm{1}_{\{T^k x < \delta\}},
		\;\frac{C}{\sqrt{n}}\sum_{k=1}^{\lfloor nt\rfloor}\mathbbm{1}_{\{T^k x\in [\delta, 1-\delta]\}},\;
	\frac{1}{n}\sum_{k=1}^{\lfloor nt\rfloor}\mathbbm{1}_{\{T^k x>  1-\delta\}} :t\geq0\bigg)
	\notag\\ \label{our-main}
	&\hspace{10mm}\underset{n\to\infty}{\overset{\cL(\mu)}{\Longrightarrow}}\;
	\big(Z_-(t),\;L(t),\;Z_+(t):t\geq0\big),
    \quad
	 \text{in $D([0,\infty), \bR^3)$},
\end{align} 
where $Z_+(t)=t-Z_-(t)$ denotes the amount of time which $B$ spends on the positive side $(0,\infty)$ up to time $t$. 
Needless to say, our result (\ref{our-main}) is a refinement of the earlier results (\ref{Aaronson}), (\ref{Ow-Sa}) and (\ref{Thaler}). 
More generally, we focus on a certain class of intermittent interval maps with two or more indifferent fixed points, and study the time evolution of the occupations near and away from these points. We show that the scaling limits are occupation times and a local time of a \emph{skew Bessel diffusion process} moving on multiray. See Table \ref{table:scaling-limit}. In addition, using the functional convergence of occupation time processes, we can also obtain those of waiting time processes for the occupations away from the indifferent fixed points.  This convergence is a functional extension of \emph{Dynkin and Lamperti type generalized arcsine laws} for waiting times \cite{T98, TZ, Z07, KZ}. 

\begin{table}[]
\centering
\begin{tabular}{|l||l|l|l|}
\hline
 & Darling--Kac type & \begin{tabular}[c]{@{}c@{}}Dynkin and\\ Lamperti type\end{tabular} & Lamperti type\\ \hhline{|=#=|=|=|}
\begin{tabular}[c]{@{}c@{}}time marginal\\ limit\end{tabular} & \begin{tabular}[c]{@{}c@{}}Mittag-Leffler \\ distribution \end{tabular}& beta distribution & \begin{tabular}[c]{@{}c@{}}Lamperti's generalized\\ arcsine distribution \end{tabular} \\ \hhline{|=#=|=|=|}
functional limit& Bessel local time & Bessel waiting time & Bessel occupation time \\ \hline
\end{tabular}
\caption{Three types of limit theorems and corresponding scaling limits}
\label{table:scaling-limit}
\end{table}

The methods of the proofs in earlier studies have been calculus of moments or double Laplace transforms. 
Instead, we adopt a method of pathwise analysis. We are inspired by the studies of diffusion processes via the It\^o excursion theory, for examples, \cite{BPY, W95, FKY, Y17}.
We utilize the first return map to analyze the statistics of the excursions wandering near the indifferent fixed points (Section \ref{sec:proof2}). Then, using Tyran-Kami\'nska's functional limit theorem \cite{Ty10a, Ty10b}, we show that the partial sum process of the excursion lengths under a certain scaling converges to a stable L\'evy process, which implies the functional convergence of occupation time processes, since the occupation times can be represented in terms of the excursion lengths (Section \ref{sec:proof}).

Let us give a quick review of earlier studies of occupation times in the context of infinite ergodic theory. See Table \ref{table:ergodic}. 
Aaronson \cite{A81} studied Darling--Kac type limit theorem. Thaler \cite{T98} studied Dynkin and Lamperti type generalized arcsine laws for waiting times. Thaler \cite{T02} studied Lamperti type generalized arcsine laws for occupation times. These three types of limit theorems were individually extended to some classes of infinite ergodic transformations by Thaler--Zweim\"uller \cite{TZ} and Zweim\"uller \cite{Z07}.
Kocheim--Zweim\"uller \cite{KZ} obtained a joint-distributional extension both of Darling--Kac type and of Dynkin and Lamperti type.
Aaronson \cite{A86} and Owada--Samorodnitsky \cite{OSb} obtained a functional extension of Darling--Kac type. In the previous study \cite{SY}, the author and Kouji Yano obtained a multiray generalization of Lamperti type. They focused on interval maps with more than two indifferent fixed points, and studied the joint-law of occupations near each of these points. 
\renewcommand{\arraystretch}{1.2}
\begin{table}[]
\begin{center}
\begin{tabular}{|l||l|l|l|}
\hline
& \begin{tabular}[l]{@{}c@{}}Darling--Kac type \\ (occupation time for\\
 subset of finite mass) \end{tabular}   
& \begin{tabular}[c]{@{}c@{}}Dynkin and
          \\ Lamperti type \\ (waiting time)
                        \end{tabular} 
           &  \begin{tabular}[c]{@{}c@{}}Lamperti type \\ (occupation time for \\
           subset of infinite mass) \end{tabular}  
\\ \hhline{|=#=|=|=|}
\multirow{2}{*}{\begin{tabular}[c]{@{}c@{}} time\\ marginal\end{tabular}}  & 
\begin{tabular}[c]{@{}l@{}}Aaronson \cite{A81}
                       \\ Thaler--Zweim\"uller \cite{TZ}
                       \\ Zweim\"uller \cite{Z07}
                        \end{tabular} 
& \begin{tabular}[c]{@{}l@{}}Thaler \cite{T98}
                          \\ T--Z \cite{TZ}
                          \\ Z \cite{Z07}
                          \end{tabular} 
& \begin{tabular}[c]{@{}l@{}}Thaler \cite{T02}
                         \\ T--Z \cite{TZ}
                          \\ Z \cite{Z07}
\end{tabular} \\ \cline{2-3} 
   & \multicolumn{2}{c|}{Kocheim--Zweim\"uller \cite{KZ} (joint-law)} 
   & \begin{tabular}[c]{@{}l@{}}Sera--Kouji Yano \cite{SY}
 	                          \\ (multiray) 
 \end{tabular}
\\ \hhline{|=#=|=|=|}
\multirow{2}{*}{functional}
       &  \begin{tabular}[c]{@{}l@{}}
       Aaronson \cite{A86}
       \\
       Owada--Samorodnitsky \cite{OSb}
       \end{tabular}
        &   &   \\ \cline{2-4} 
 & \multicolumn{3}{c|}{\textbf{This paper} (joint-law, multiray) }   \\ \hline
\end{tabular}
\caption{This paper and earlier studies in the context of infinite ergodic theory}
\label{table:ergodic}
\end{center}
\end{table}
\renewcommand{\arraystretch}{1}

\renewcommand{\arraystretch}{1.2}
\begin{table}[]
\begin{center}\begin{tabular}{|l||l|l|l|}
\hline
&  \begin{tabular}[c]{@{}c@{}}Darling--Kac type\\ (local time) \end{tabular}    
& \begin{tabular}[c]{@{}c@{}}Dynkin and
          \\ Lamperti type \\ (waiting time)
                        \end{tabular} 
           &  \begin{tabular}[c]{@{}c@{}}Lamperti type \\ (occupation time)  \end{tabular}
\\ \hhline{|=#=|=|=|}
\multirow{3}{*}{\begin{tabular}[c]{@{}c@{}} time \\marginal \end{tabular}}  & 
\begin{tabular}[c]{@{}l@{}}L\'evy \cite{Le, Le48}
                   \\ Darling--Kac \cite{DK}
                   \\ Karlin--McGregor \cite{KaMc}
                   \\ Kasahara \cite{Kas75, Kas76}
                        \end{tabular} 
& \begin{tabular}[c]{@{}l@{}} L\'evy \cite{Le, Le48}
                          \\ Dynkin \cite{Dyn61}
                          \\ Lamperti \cite{La-b}
                          \end{tabular} 
& \begin{tabular}[c]{@{}l@{}} L\'evy \cite{Le, Le48}
                         \\ Lamperti \cite{La}
                      \\ Watanabe \cite{W95}
                      \\ Yuko Yano \cite{Y17}
\end{tabular} \\ \cline{2-3} 
   & \multicolumn{2}{c|}{Port \cite{Port} (joint-law)} 
   & \begin{tabular}[c]{@{}l@{}}  (multiray) 
 \end{tabular}
\\ 
\cline{2-4}
& \multicolumn{3}{c|}{Barlow--Pitman--Yor \cite{BPY} (joint-law, multiray)} 
\\
\hhline{|=#=|=|=|}
\multirow{2}{*}{functional}
       & Bingham \cite{B71} & Lamperti \cite{La62} & 
       \begin{tabular}[c]{@{}l@{}}
       Fujihara--Kawamura--\\Yuko Yano \cite{FKY}
       \end{tabular}
        \\ \cline{2-4} 
 & \multicolumn{3}{c|}{(Donsker \cite{Don}, Stone \cite{St}, Kasahara \cite{Kas75}) }   \\ \hline
\end{tabular}
\caption{Earlier studies in the context of null-recurrent Markov processes}
\label{table:prob}
\end{center}

\end{table}
\renewcommand{\arraystretch}{1}

The above-mentioned studies are analogues of limit theorems for null-recurrent Markov processes. See Table \ref{table:prob}. On the one hand, L\'evy \cite{Le, Le48} constructed the Brownian local time, and showed that its one-dimensional marginal law is a half Gaussian distribution. Darling--Kac \cite{DK} studied the occupation times of general null-recurrent Markov processes on a subset having finite mass. Darling--Kac type limit theorems have been further developed, for example by  Karlin--McGregor \cite{KaMc} for birth-and-death processes, by Kasahara \cite{Kas75, Kas76} for diffusion processes. Bingham \cite{B71} studied functional extension of Darling--Kac type. 
On the other hand, L\'evy \cite{Le, Le48} focused on one-dimensional Brownian motion and simple symmetric random walk, and obtained the well-known arcsine laws for waiting times and occupation times. The arcsine law for waiting times was generalized by Dynkin \cite{Dyn61} and Lamperti \cite{La-b}. Lamperti \cite{La62} further obtained a functional extension of them. 
Port \cite{Port} obtained a joint-distributional extension both of Darling--Kac type and of Dynkin and Lamperti type at the same time.
The arcsine law for occupation times was generalized by Lamperti \cite{La}. 
Barlow--Pitman--Yor \cite{BPY} focused on skew Bessel diffusion processes moving on multiray, and investigated the joint-distribution of the local times, of the waiting times, and of the occupation times, at the same time. Watanabe \cite{W95} and Yuko Yano \cite{Y17} further developed Lamperti type generalized arcsine laws for one-dimensional or multiray diffusion processes.
Fujihara--Kawamura--Yuko Yano \cite{FKY} obtained a functional extension of Lamperti \cite{La}.
We remark that functional limit theorems for trajectories were also obtained by Donsker \cite{Don} in the setting of non-heavy tailed random walks, by Stone \cite{St} and Kasahara \cite{Kas75} in the setting of one-dimensional diffusion processes.

The present paper is organized as follows. In Section \ref{sec:notation}, we set up notations and state our assumptions in a general setting. In Section \ref{sec:bessel}, we recall the definition and known results for skew Bessel diffusion processes.  
In Section \ref{sec:main}, we formulate our main result in the general setting and then apply it to interval maps with indifferent fixed points.
In Section \ref{sec:Skorokhod}, we recall the Skorokhod $J_1$-topology and prepare auxiliary results. 
The proofs of our general limit theorem and of its application are given in Sections \ref{sec:proof} and \ref{sec:proof2}, respectively. 
In Appendices, we recall several facts about uniformly expanding interval maps and functional stable limit theorem for stationary sequences.

\subsection*{Acknowledgements}

The author are grateful for many helpful comments  provided by Kouji Yano. The author also thanks Takuma Akimoto for drawing the author's attention to connections between interval maps with indifferent fixed points and intermittent phenomena.

\section{Notations and assumptions in a general setting}
\label{sec:notation} 

In this section, we will recall the notions of CEMPT, wandering rates, regular variation and exponentially continued fraction mixing. We will formulate certain assumptions (Assmptions \ref{ass:d-ray}, \ref{ass:reg-var} and \ref{ass:mixing}) for our functional limit theorem. These assumptions are satisfied in the setting of certain classes of null-recurrent Markov chains (Remark \ref{rem:markov}) and of intermittent interval maps (Subsection \ref{subsec:int} and Section \ref{sec:proof2}).
 
Let $(X,\cB,\mu)$ be a $\sigma$-finite measure space with $\mu(X)=\infty$. 
Let $T:(X,\cB,\mu) \to (X,\cB,\mu)$ be a conservative, ergodic, measure preserving transformation, which will be abbreviated as \textit{CEMPT}. 
Equivalently, we assume that $\mu\circ T^{-1}=\mu$ and, for any $A\in\cB$ with $\mu(A)>0$,
\begin{align*}
	\sum_{n\geq0}\mathbbm{1}_{\{T^n x \in A\}}=\infty,
	\quad
	\text{$\mu$-a.e.$x$,}
\end{align*}
where $\mathbbm{1}_{\{\cdot\}}$ denotes the indicator function.
For the details,  see \cite[Chapter 1]{A97}. 
We always use the summation signs to denote unions of disjoint sets; for example, $A_1+A_2$, $\sum_j A_j$ and so on.  
Following \cite{TZ, Z07, SY}, we will impose the following assumption from now on.

\begin{Ass}[Dynamical separation]\label{ass:d-ray}
Let $d\geq1$ be a positive integer. The state space $X$ can be decomposed into $X = \sum_{j=1}^d A_j+Y$ for $A_1,\dots,A_d,Y\in\mathcal{B}$ with $\mu(A_j)=\infty$ ($j=1,\dots,d$) and $\mu(Y)\in(0,\infty)$.
Furthermore, $Y$ \emph{dynamically separates $A_1,\dots,A_d$}, that is, the condition [$x \in A_i$ and $T^n x\in A_j$ for some 
$i \neq j$ and some $n\geq1$] implies the existence of some $k\in\{1,\dots,n-1\}$ for which $T^{k}x \in Y$.
\end{Ass}

Roughly speaking, $A_1,\dots,A_d$ play roles of rays, while $Y$ plays a role of a junction; the orbit cannot pass from $A_i$ to $A_j$ ($i\neq j$) without visiting $Y$. For $A\in\cB$ and $t\geq0$, we define the measurable function $S_A(t):X\to [0,\infty)$ by  
\begin{align}
\label{sA}
S_A(t)(x)&:=\sum_{k=1}^{\lfloor t\rfloor} \mathbbm{1}_{\{T^k x \in A\}},
              \quad x\in X.
\end{align}
In other words, $S_A(t)$ ($n\in\bN$) denotes the amount of time which the orbit  spends on $A$ up to  time $\lfloor t\rfloor$.
Birkhoff's pointwise ergodic theorem implies
\begin{align}\label{Birkhoff}
	\frac{1}{t}S_{Y}(t)
	\underset{t\to\infty}{\to}
	0
	\quad \bigg(\text{or equivalently,}\;\;
	\frac{1}{t} S_{A_1+\dots+A_d}(t)
	\underset{t\to\infty}{\to}
	1\bigg),
	\quad
	\text{$\mu$-a.e.}
\end{align}
We will denote by $(w(n))_{n\geq0}$ the \emph{wandering rate} of $Y$, and by $(w_j(n))_{n\geq0}$ the \emph{wandering rate of $Y$ starting from $A_j$}, i.e.,
\begin{align}
 w(n)&:=
 \mu\bigg(\bigcup_{k=0}^{n-1}T^{-k}Y\bigg),
 \quad
 n\geq0,
 \\
 w_j(n)&:=
 \mu\bigg(\bigcup_{k=0}^{n-1}(T^{-k}Y\cap A_j)\bigg),
 \quad n\geq0,\;j=1,\dots,d.	
\end{align}

Let $f,g:(0,\infty)\to(0,\infty)$ be measurable functions. For a constant $c\in [0,\infty]$, we will write $f(x)\underset{x\to x_0}{\sim} c g(x)$ if it holds that $\lim_{x\to x_0}f(x)/g(x)=c$.
We note that $c g(x)$ has only a symbolic meaning if $c=0$ or $\infty$. See \cite[p.$\:$xix]{BGT}.  
Let $\alpha\in\bR$. We will write $f\in\cR_\alpha(\infty)$ if  $f$ is \emph{regularly varying} of index $\alpha\in\bR$ at $\infty$, that is, $f(\lambda r)\underset{r\to\infty}{\sim} \lambda^\alpha f(r)$ for each $\lambda>0$. Similarly, we will write $f\in\cR_\alpha(0+)$ if $f$ is regularly varying of index $\alpha$ at $0$.
Let $(a_n)_{n\geq0}$ be a $(0,\infty)$-valued sequence. We will write $a_n\in \cR_\alpha(\infty)$ if the function [$r\mapsto a_{\lfloor r \rfloor}$] is regularly varying of index $\alpha$ at $\infty$.
 For basic discussions of regular variation, we refer the reader to Bingham--Goldie--Teugels \cite[Chapter 1]{BGT}.

The following assumption is essentially needed for the existence of non-trivial limit of $S_{A_1},\dots,S_{A_d}, S_Y$, as shown in earlier studies. See for example \cite[Theorems 2.7 and 2.8]{SY}. 

\begin{Ass}[Regular variations of wandering rates]\label{ass:reg-var}
For constants $\alpha\in(0,1)$, $\beta=(\beta_1,\dots,\beta_d)\in[0,1]^d$ with $\sum_{j=1}^d \beta_j=1$, it holds that 
\begin{align}
    &
    w(n) \in \cR_{1-\alpha}(\infty),
            \label{eq:reg}
   \\ 
    &
    w_j(n)\underset{n\to\infty}\sim \beta_j w(n), \quad j=1,\dots,d. 
    \label{eq:bal}
\end{align}	
\end{Ass}

From now on, suppose that Assumption \ref{ass:reg-var} is satisfied.
Let us define measurable functions $\varphi:X\to \bN\cup\{\infty\}$ and $\ell_j:X\to \bN\cup\{\infty\}$
 by   
\begin{align}
   \varphi(x)
   &:=\min\{k\geq1:T^k x\in Y\},
   \quad x\in X,
   \\
   \ell_j(x)
   &:=
   S_{A_j}(\varphi(x))(x)
   =\sum_{k=1}^{\varphi(x)}\mathbbm{1}_{\{T^kx\in A_j\}},
   \quad x\in X,
\end{align}
where it is understood that $\min\emptyset=\infty$.
In other words, $\varphi(x)$ denotes the first return time for $Y$, and $\ell_j(x)$ denotes the length of time spent in $A_j$ by the first excursion $(T^k x:1\leq k\leq \varphi(x))$ away from $Y$. 
Since $T$ is a CEMPT, we have $\varphi, \ell_j<\infty$, $\mu$-a.e. By Assumption \ref{ass:d-ray}, we have
\begin{align*}
	\{\ell_j=n\}
	&=
	T^{-1}A_j\cap\{\varphi=n+1\}
	\\
	&=
	\{x\in X: Tx,\dots,T^{n}x\in A_j \;\;\text{and}\;\; T^{n+1}x\in Y\},
	\quad n\geq1.
\end{align*}
Let us define a probability measure $\mu_Y$ on $(X,\cB)$ by
\begin{align}
   \mu_Y (A)
   :=\mu(A\cap Y)/\mu(Y), \quad A\in\cB.	
\end{align}
As shown in the proof of \cite[Lemma 6.2]{TZ} (see also \cite[Lemma 4.1]{SY}), the difference of $w$ is equal to the tail measure of $\varphi$, while the difference of $w_j$ is equal to the tail measure of $\ell_j$:
\begin{align}
  w(n+1)-w(n)
  &=
  \mu(Y)\mu_Y[\varphi>n],
  \quad n\geq0,
  \label{eq:difference-1}
  \\
  w_j(n+1)-w_j(n)
  &=
  \mu(Y)\mu_Y[\ell_j\geq n],
  \quad n\geq1,\;j=1,\dots,d.
  \label{eq:difference-2}	
\end{align}
Hence Karamata's Tauberian theorem implies the following lemma:

\begin{Lem}[Regular variations of excursion lengths]
Let $T$ be a CEMPT on $(X,\cB,\mu)$. Suppose that Assumptions \ref{ass:d-ray} and \ref{ass:reg-var} are satisfied. Then it holds that
\begin{align}
	&\mu_Y[\varphi > n]\underset{n\to\infty}{\sim} \frac{1-\alpha}{\mu(Y)}\frac{w(n)}{n}\in \cR_{-\alpha}(\infty), 
	\\
	&\mu_Y[\ell_j\geq n] 
	\underset{n\to\infty}{\sim}
	\beta_j \mu_Y[\varphi>n],
	\quad
	j=1,\dots,d.
\end{align} 
\end{Lem}

Set
\begin{align}
   \ell (x)
   &:=
   (\ell_1(x),\dots,\ell_d(x)),
   \quad x\in X.
\end{align}
We will denote by $T_Y:X\to Y$ the first return map for $Y$:
\begin{align}
	T_Y x:= T^{\varphi(x)}x,
	\quad x\in X.
\end{align}
The map $T_Y$ is a CEMPT on the probability space $(X,\cB, \mu_Y)$. Therefore, the sequence $(\ell\circ T_Y^n:n\geq 0)$ is strictly stationary with respect to $\mu_Y$. Note that $(\ell_j \circ T_Y^n)(x)$ is the length of time spent in $A_j$ by the $(n+1)$th excursion $(T^k (T_Y^n x): 1 \leq k \leq \varphi_Y (T^n_Yx))$ away from $Y$.
\begin{Ex}
Assume that 
\begin{align*}
  (T^k x)_{k=0}^6
  \in
  Y\times A_1\times Y \times A_2 \times A_2 \times Y \times Y. 
\end{align*}
Then we have 
	$\ell(x)=(1,0,\dots,0)$,
	$(\ell\circ T_Y) (x)=(0,2,\dots,0)$
	and $(\ell\circ T^2_Y) (x)= (0,0,\dots,0)$.
\end{Ex}

For $0\leq n\leq m \leq \infty$, we define the sub-$\sigma$-field $\cF_n^m\subset\cB$ by
\begin{align}
\cF_n^m:=\sigma\{\ell\circ T_Y^{k}:n\leq k \leq m\}.
\end{align}
That is, $\cF_n^m$ denotes the sub-$\sigma$-field generated by the $(n+1)$th, ($n+2$)th, $\dots$, and $(m+1)$th excursion lengths.
Finally, we will impose the following assumption:

\begin{Ass}[Local dependence of excursion lengths]\label{ass:mixing}
The sequence $(\ell\circ T_Y^n:n\geq0)$ of the excursion lengths is \emph{exponentially continued fraction mixing} with respect to $\mu_Y$. That is, there exist some constants $C\in(0,\infty)$ and $\theta\in(0,1)$ such that, for any $k,n\geq0$, $A\in\cF_0^k$ and $B\in\cF_{k+n}^\infty$,
\begin{align}
 |\mu_Y(A\cap B)-\mu_Y(A)\mu_Y(B)|
 \leq
 C\theta^n
 \mu_Y(A)\mu_Y(B).	
\end{align}
\end{Ass}

\begin{Rem}
Suppose that Assumptions \ref{ass:d-ray}, \ref{ass:reg-var} and \ref{ass:mixing} are satisfied. In addition, assume that $\cF_0^\infty=\cB\cap Y$. Then $Y$ is a Darling--Kac set, and hence $T$ is pointwise dual ergodic with normalizing constants
\begin{align}
	a_n:=\sum_{k=1}^n\frac{\mu(Y\cap T^{-k}Y)}{(\mu(Y))^2}
	   \underset{n\to\infty}{\sim}
	   \frac{1}{\Gamma(2-\alpha)\Gamma(1+\alpha)}\frac{n}{w(n)}\in\cR_\alpha(\infty).
\end{align}
See \cite[\S3.7-3.8]{A97} and \cite[Section 1]{A86} for the details. Consequently,  the assumption of \cite[Theorem 1]{A81} and \cite[Theorem 6.1]{OSb} is satisfied.
\end{Rem}

For $A\in\cB$ and $t\geq0$, we will denote by $G_A(t)$ the last visit time for $A$ before the time $t$, while we will denote by $D_A(t)$ the first visit time  for $A$ after the time $t$:
\begin{align}
  G_A(t)(x)&:=\max\{1\leq k\leq t: T^k x \in A\}, \quad x\in X,
  \label{def:G_A}
 \\
  D_A(t)(x)&:=\min\{k> t :T^k x \in A\}, \quad x\in X,
  \label{def:D_A}
\end{align}
where $\max\emptyset=0$ and $\min\emptyset=\infty$.
Here we imitate the notation of the last exit decomposition of Markov processes (see for example \cite{G79, RY, S97, SVY}), and our notation differs from that of \cite{T98, TZ, Z07, KZ}.

Let $U$ be a Polish space, $\nu_0$ a probability measure on $(X,\cB)$, and $(F_n)_{n\geq0}$ a sequence of $U$-valued measurable functions defined on $(X,\cB)$.
Let  $\zeta$ be a $U$-valued random variable defined on a probability space $(\Omega,\cF,\bP)$. 
We will write [$F_n \overset{\nu_0}{\Longrightarrow} \zeta$, in $U$] if the sequence of pushforward probability measures $(\nu_0\circ F_n^{-1})_{n\geq0}$ converges weakly to the law $\bP[\zeta\in\cdot]$ of $\zeta$.
We say that the $F_n$ \textit{converge to $\zeta$ strongly in distribution} (with respect to $\mu$) if, for any probability measure $\nu \ll \mu$ on $(X,\cB)$, the convergence [$F_n \overset{\nu}{\Longrightarrow} \zeta$, in $U$] holds.
Following \cite{TZ, Z07, Z07b, KZ}, we will denote by [$F_n \overset{\cL(\mu)}{\Longrightarrow} \zeta$, in $U$] the strong distributional convergence. 

We will write $C([0,\infty),\bR^d)$ for the space of all continuous functions $w:[0,\infty)\to \bR^d$. The space $C([0,\infty),\bR^d)$ is endowed with the Polish topology of uniform convergence on compact subsets of $[0,\infty)$.
We will denote by $D([0,\infty),\bR^d)$ the space of $\bR^d$-valued c\`adl\`ag functions $w:[0,\infty)\to\bR^d$. We equip $D([0,\infty),\bR^d)$ with the \emph{Skorokhod $J_1$-topology}, which is a Polish topology. See Section \ref{sec:Skorokhod} for the Skorokhod $J_1$-topology.

We are interested in a non-trivial scaling limit of the time evolution
\begin{align*}
	\Big((S_{A_j}(t))_{j=1}^d, S_Y(t), G_Y(t), D_Y(t):t\geq0\Big),
\end{align*}
in the sense of strong distributional convergence in $D([0,\infty),\bR^{d+3})$.
In Section \ref{sec:main}, we will state that the scaling limit is  the joint of occupation times, local time and waiting times of a skew Bessel diffusion process on multiray.

\section{Skew Bessel diffusion processes on multiray}\label{sec:bessel}

Following Barlow--Pitman--Yor \cite{BPY}, we will define skew Bessel diffusion processes on multiray and recall known results for them.  For basic properties of Bessel diffusion processes reflected at the origin, see for instance  Revuz--Yor \cite[Chapter XI]{RY}, Barlow--Pitman--Yor \cite[Section 2]{BPY} and Donati-Martin--Roynette--Vallois--Yor \cite{DRVY}.
For a deeper discussion of diffusion processes on multiray, we refer the reader to Yano \cite[Section 2]{Y17}. 
For basic discussions of Poisson point processes and the It\^o excursion theory, see for instance It\^o \cite{I15} and Revuz--Yor \cite[Chapter XII]{RY}.

Let $(B(t):t\geq0)$ be a standard one-dimensional Brownian motion. For $\alpha\in(0,1)$, let us consider the stochastic differential equation
\begin{align}\label{BESQ}
   X^{(\alpha)}(0)=0,
   \quad \rd X^{(\alpha)}(t)=2 \sqrt{X^{(\alpha)}(t)}\rd B(t)+(2-2\alpha)\rd t,
   \quad t\geq0.	
\end{align}
The equation \eqref{BESQ} has a unique strong solution $(X^{(\alpha)}(t):t\geq0)$ which is called the \emph{square of $(2-2\alpha)$-dimensional Bessel diffusion process} started and instantaneously reflected at the origin. 
The square root 
\begin{align}
R^{(\alpha)}=\Big(R^{(\alpha)}(t)=\sqrt{X^{(\alpha)}(t)}:t\geq0\Big)
\end{align}
is called the \emph{$(2-2\alpha)$-dimensional Bessel diffusion process}  started and instantaneously reflected at the origin.
Their transition probability densities can be written in terms of the Bessel functions. That is the reason why they are called Bessel diffusion processes. 
The Bessel diffusion processes under their natural scale appear as the scaling limits of $[0,\infty)$-valued generalized diffusion processes (see \cite{St} and \cite[Section 5]{Kas75}).
In the special case of $\alpha=1/2$, this process is nothing else but the reflecting Brownian motion, that is, the absolute value of a standard one-dimensional Brownian motion.

Let us denote by $(L^{(\alpha)}(t,x):t,x\geq0)$ the local time of $R^{(\alpha)}$.
More specifically, the map $[0,\infty)^2\ni(t,x)\mapsto L^{(\alpha)}(t,x)\in [0,\infty)$ is jointly continuous, and for any bounded and measurable function $f:[0,\infty)\to[0,\infty)$, the following occupation-time formula holds:
\begin{align*}
  \int_0^t f(R^{(\alpha)}(s))\rd s
  =
  C^{(\alpha)}\int_0^\infty f(x)L^{(\alpha)}(t,x) x^{1-2\alpha} \rd x,
  \quad
  t\geq0,	
\end{align*}
where $C^{(\alpha)}:=2^{\alpha}\Gamma(\alpha)/\Gamma(1-\alpha)$. We will write $L^{(\alpha)}=(L^{(\alpha)}(t):t\geq0)$ for the local time of $R^{(\alpha)}$ at the origin (in the Blumenthal--Getoor normalization in the sense of \cite[VI.(45.5)]{RW2}):
\begin{figure}
\begin{minipage}{0.5\hsize}
\centering
	\includegraphics[width=0.8\linewidth]{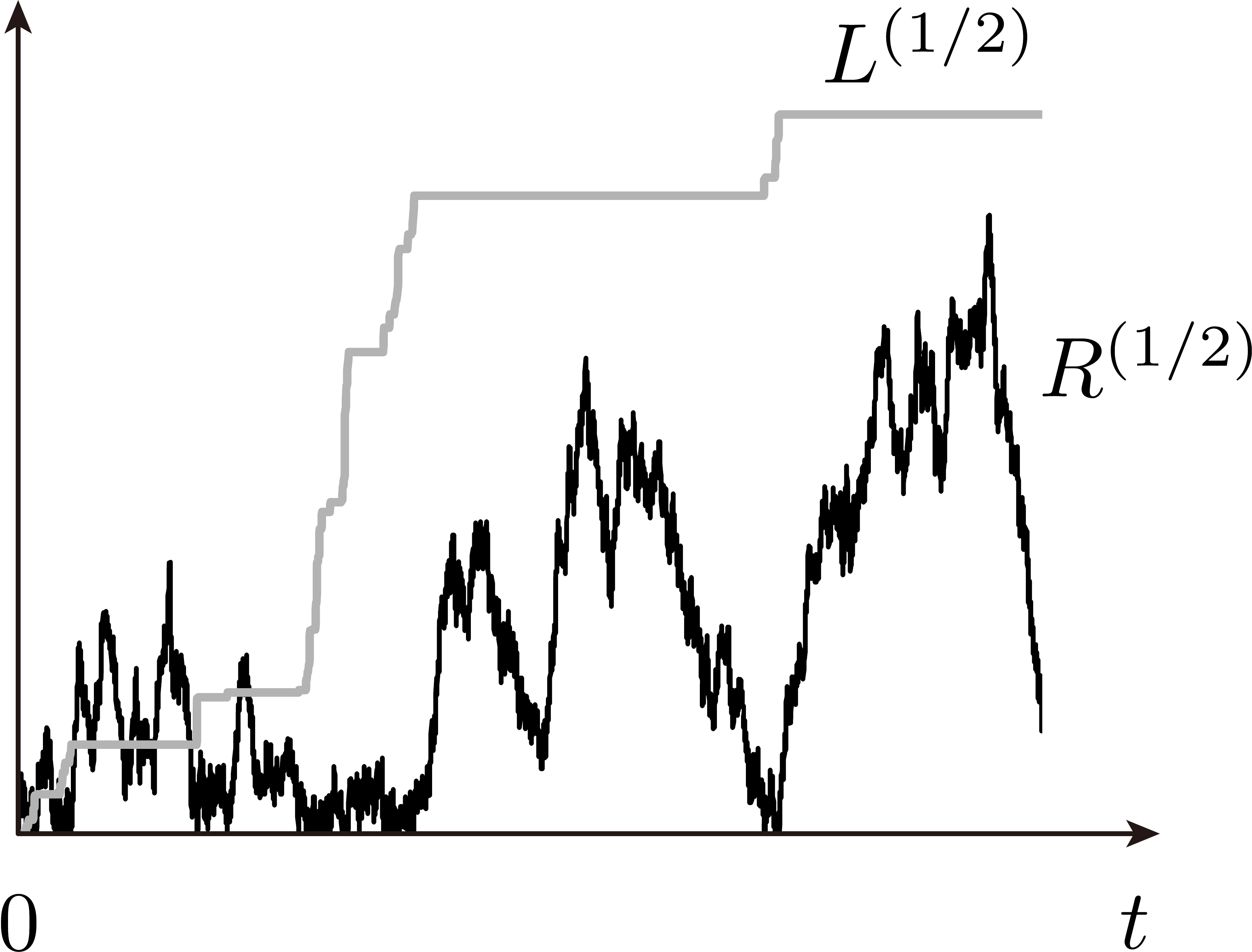}
	\caption{Reflecting Brownian motion $R^{(1/2)}$ and its local time $L^{(1/2)}$ at the origin}
	\label{fig:loc}
\end{minipage}
\hspace{0.1\hsize}
\begin{minipage}{0.35\hsize}
\centering
	\includegraphics[width=0.7\linewidth]{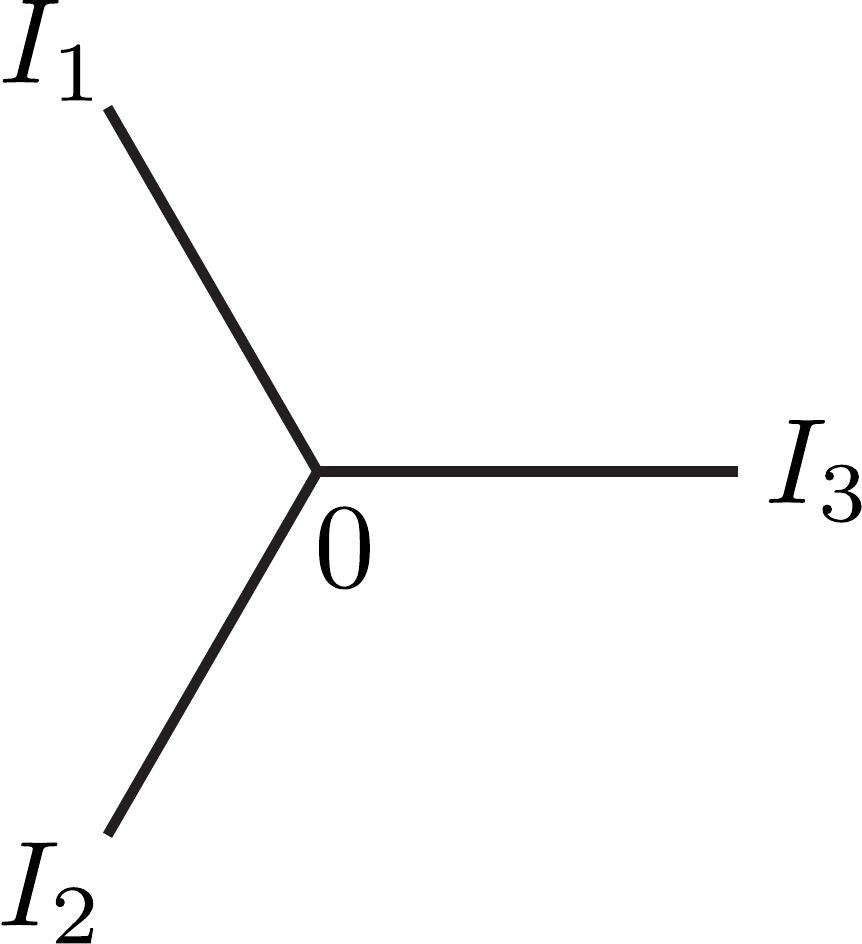}
	\caption{Multiray $I_1\cup I_2\cup I_3$ in the case of $d=3$}
	\label{fig:multiray}	
\end{minipage}
\end{figure}
\begin{align}
  L^{(\alpha)}(t)
  :=
  L^{(\alpha)}(t,0)
  =
  \lim_{\varepsilon\downarrow 0}
  \frac{2-2\alpha}{C^{(\alpha)}\varepsilon^{2-2\alpha}}
  \int_0^t \mathbbm{1}_{\{R^{(\alpha)}(s)\leq\varepsilon\}}\rd s,
  \quad
  t\geq0.
\end{align}
The process $L^{(\alpha)}$ increases only at the zeros of $R^{(\alpha)}$. See Figure \ref{fig:loc}.
The $L^{(\alpha)}$ is a Mittag-Leffler process of order $\alpha$. Its one-dimensional distributions are characterized by
\begin{align*}
	\bE[\exp(-\lambda L^{(\alpha)}(t))]
	=
	\sum_{n\geq0}
	\frac{(-\lambda t^{\alpha})^n}{\Gamma(1+n\alpha)},
	\quad
	\lambda,t\geq0.
\end{align*}
The finite-dimensional distributions of $L^{(\alpha)}$ are characterized by \cite[Propositions 1(a) and 1(b)]{B71}.
We will denote by $\eta^{(\alpha)}=(\eta^{(\alpha)}(s):s\geq0)$ the inverse local time of $R^{(\alpha)}$ at the origin:
\begin{align}
   \eta^{(\alpha)}(s)
   :=
   \big(L^{(\alpha)}\big)^{-1}(s)
   =
   \inf\{t>0:L^{(\alpha)}(t)>s\},
   \quad
   s\geq0.	
\end{align}
Then $\eta^{(\alpha)}$ is an $\alpha$-stable subordinator with Laplace transform
\begin{align*}
  \bE[\exp(-\lambda \eta^{(\alpha)}(s))]
  =
  \exp(-\lambda^{\alpha}s)
  =
  \exp\bigg(-s\int_0^\infty (1-e^{-\lambda r})\frac{\alpha r^{-1-\alpha}}{\Gamma(1-\alpha)}\rd r\bigg),
  \quad
  \lambda,s\geq0.	
\end{align*}
We know that the zero set of $R^{(\alpha)}$ coincides with the closure of the image set of $\eta^{(\alpha)}$:
\begin{align}\label{zeros-inv-local}
	\{t\geq0: R^{(\alpha)}(t)=0\}
	=
	\bar{\{\eta^{(\alpha)}(s):s\geq0\}}.
\end{align}
For $w=(w(t):t\geq0)\in D([0,\infty),\bR^d)$, set $w(t-):=\lim_{u\uparrow t}w(u)$ and $\Delta w(t):= w(t)-w(t-)$.
Let us denote by $e^{(\alpha)}=(e^{(\alpha)}_s=(e^{(\alpha)}_s(t):t\geq0):s\geq0)$ the excursion point process of $R^{(\alpha)}$ away from the origin for the normalization $L^{(\alpha)}$:
\begin{align}
	e^{(\alpha)}_s (t)
	:=
	\begin{cases}
	R^{(\alpha)}(t+\eta^{(\alpha)}(s-)) \mathbbm{1}_{\{t\leq \Delta\eta^{(\alpha)}(s)\}},
	& \text{if $\Delta\eta^{(\alpha)}(s)>0$,}
	\\
	0,
	& \text{otherwise.}
	\end{cases}
\end{align}
Recall that $e^{(\alpha)}$ is a Poisson point process with values in $C([0,\infty),[0,\infty))$.
Note that $\eta^{(\alpha)}$ coincides with the partial sum process of the lifetimes of $e^{(\alpha)}$:
\begin{align}
	\eta^{(\alpha)}(s)=\sum_{u\leq s}\Delta \eta^{(\alpha)}(u)=\sum_{u\leq s} \inf\{t>0:e_u^{(\alpha)}(t)=0\},
	\quad
	s\geq0.
\end{align}
We will denote by $\bn^{(\alpha)}$ the It\^o characteristic measure of the excursion point process $e^{(\alpha)}$.
By disintegreting $\bn^{(\alpha)}$ with respect to the lifetime, we have  
\begin{align*}
	\bn^{(\alpha)}(\cdot)=\int_0^\infty \bP_{0,0}^{2+2\alpha,r}(\cdot) \frac{\alpha r^{-1-\alpha}}{\Gamma(1-\alpha)}\rd r, 
\end{align*}
where $\bP_{0,0}^{2+2\alpha,r}$ denotes the law of the $(2+2\alpha)$-dimensional Bessel bridge from the origin to itself over $[0,r]$ in the sense of \cite[Chapter XI, \S3]{RY}. 

For $z\in\bC$ with $|z|=1$, we define the map $\sigma_{z}:C([0,\infty),\bC)\to C([0,\infty),\bC)$ by
\begin{align*}
  \sigma_z((w(t):t\geq0))
  =
  (w(t)z:t\geq0),
\end{align*}
that is, $\sigma_z$ is a rotation about the origin. 
Note that the image measure $\bn^{(\alpha)}\circ \sigma_z^{-1}$ is the It\^o characteristic measure of the excursion point process of $\sigma_z(R^{(\alpha)})$ away from the origin.
Let $d\geq1$ be a positive integer. Set
\begin{align*}
    z_j
    &:= \exp\big(2\pi j\sqrt{-1}/d\big),
    \quad j=1,\dots,d,
    \\
	I_j
	&:=\{rz_j\in \bC:r\geq0\},
	\quad
	j=1,\dots,d.
\end{align*}
See Figure \ref{fig:multiray}. Let $\beta=(\beta_1,\dots,\beta_d)\in[0,1]^d$ be a constant with $\sum_{j=1}^d \beta_j=1$. Let us denote by $e^{(\alpha,\beta)}=(e^{(\alpha,\beta)}_s=(e^{(\alpha,\beta)}_s(t):t\geq0):s\geq0)$ the Poisson point process with values in $C([0,\infty), \bigcup_{j=1}^d I_j)$ and with It\^o characteristic measure
	$\sum_{j=1}^d \beta_j\big(\bn^{(\alpha)} \circ \sigma_{z_j}^{-1}\big)$. Set
\begin{align}
	\eta^{(\alpha,\beta)}(s):=\sum_{u\leq s} \inf\{t>0:e_u^{(\alpha,\beta)}(t)=0\},
	\quad s\geq0.
\end{align}
That is, $\eta^{(\alpha,\beta)}$ denotes the partial sum process of the lifetimes of $e^{(\alpha,\beta)}$. By the definition, we have $|e^{(\alpha,\beta)}|=((|e^{(\alpha,\beta)}_s(t)|:t\geq0):s\geq0)\dist e^{(\alpha)}$ and hence  $\eta^{(\alpha,\beta)}\dist \eta^{(\alpha)}$.

\begin{Def}[Skew Bessel diffusion process]\label{def:skew-bes}
 For $\alpha\in (0,1)$ and $\beta=(\beta_1,\dots,\beta_d)\in[0,1]^d$ with $\sum_{j=1}^d\beta_j=1$, we will denote by $R^{(\alpha,\beta)}=(R^{(\alpha,\beta)}(t):t\geq0)$ the process pieced together from the excursion point processes $e^{(\alpha,\beta)}$, that is,
 \begin{align}
 	R^{(\alpha,\beta)}(t)
 	:=
 	\begin{cases}
 	e^{(\alpha,\beta)}_s(t-\eta^{(\alpha,\beta)}(s-)),
 	& \text{if $t\in(\eta^{(\alpha,\beta)}(s-), \eta^{(\alpha,\beta)}(s))$ for some $s\geq0$,}
 	\\
 	0,
 	&\text{otherwise.}	
 	\end{cases}
 \end{align}
The process $R^{(\alpha,\beta)}$ is called the $(\bigcup_{j=1}^d I_j)$-valued \emph{skew Bessel diffusion process} of dimension $(2-2\alpha)$, with skewness parameter $\beta$, started at the origin.	
\end{Def}

Note that $|R^{(\alpha,\beta)}|\dist R^{(\alpha)}$.
Roughly speaking, every time the process $R^{(\alpha,\beta)}$ reaches to the origin, it chooses randomly a ray $I_j$ from rays $I_1,\dots,I_d$ with probability $\beta_j$, and then it moves like the $I_j$-valued diffusion process $\sigma_{z_j}(R^{(\alpha)})$ until it returns again to the origin.
It is a multiray diffusion process on $\bigcup_{j=1}^d I_j$ in the sense of \cite[Section 2]{Y17}  (see \cite[Remark 2.2]{Y17}).

\begin{Rem}\label{rem:skew-bes2}
In the case of $\alpha=1/2$, the process $R^{(\alpha,\beta)}$ is also called the \emph{Walsh Brownian motion} or the \emph{Brownian spider}. In the case of $d=2$ and $\alpha=1/2$, it is also called \emph{skew Brownian motion}. See \cite{HS, BPYa, Lej} and the references therein. In the special case of $d=2$ and $\alpha=\beta_1=\beta_2=1/2$, the process $R^{(\alpha,\beta)}$ is nothing else but the standard one-dimensional Brownian motion. 
\end{Rem}

Set
\begin{align}
	\label{def:bes-occ}
	Z_j^{(\alpha,\beta)}(t)
	:=
	\int_0^t \mathbbm{1}_{\{R^{(\alpha,\beta)}(s)\in I_j\}}\rd s,
	\quad
	t\geq0,
	\;j=1,\dots,d,
	\\
	\label{def:bes-loc}
	L^{(\alpha,\beta)}(t)
	:=
  \lim_{\varepsilon\downarrow 0}
  \frac{2-2\alpha}{C^{(\alpha)}\varepsilon^{2-2\alpha}}
  \int_0^t \mathbbm{1}_{\{|R^{(\alpha,\beta)}(s)|\leq\varepsilon\}}\rd s,
  \quad
  t\geq0.
\end{align}
That is, $Z_j^{(\alpha,\beta)}(t)$ denotes the amount of time which $R^{(\alpha,\beta)}$ spends on $I_j$ up to time $t$, and $L^{(\alpha,\beta)}(t)$ denotes the local time of $R^{(\alpha,\beta)}$ at the origin up to time $t$. 
Note that $\eta^{(\alpha,\beta)}=(L^{(\alpha,\beta)})^{-1}$ and $L^{(\alpha,\beta)}\dist L^{(\alpha)}$.
%

\begin{Thm}[{\cite[Theorem 1]{BPY}}]\label{thm:BPY}
For $t>0$, it holds that
 \begin{align}
 	\bigg(\bigg(\frac{Z^{(\alpha,\beta)}_j(t)}{t}\bigg)_{j=1}^d,\;
 	      \frac{L^{(\alpha,\beta)}(t)}{t^\alpha} \bigg)
 	\dist
 	\bigg( \bigg(\frac{\xi_j}{\sum_{i=1}^d\xi_i}\bigg)_{j=1}^d, \;
     \frac{1}{(\sum_{i=1}^d\xi_i)^\alpha} \bigg),
    \quad\text{in $\bR^{d+1}$},
    \label{joint-dist}
\end{align}
where $\xi_1,\dots,\xi_d$ denote independent $[0,\infty)$-valued random variables with the one-sided $\alpha$-stable distributions characterized by
\begin{align}\label{xi-lap}
  \bE[\exp(-\lambda \xi_j)]
  =
  \exp(-\lambda^\alpha \beta_j),
  \quad
  \lambda\geq0,
  \;
  j=1,\dots,d.	
\end{align}	
\end{Thm}

Recall that the joint-law of $(\xi_j/\sum_{i=1}^d \xi_i)_{j=1}^d$ is a multi-dimensional version of Lamperti's generalized arcsine distributions (see for example \cite[Subsection 2.2]{SY}), and  the law of $(\sum_{i=1}^d\xi_i)^{-\alpha}$ is a Mittag-Leffler distribution of order $\alpha$.

\begin{figure}
\centering
\includegraphics[width=0.6\linewidth]{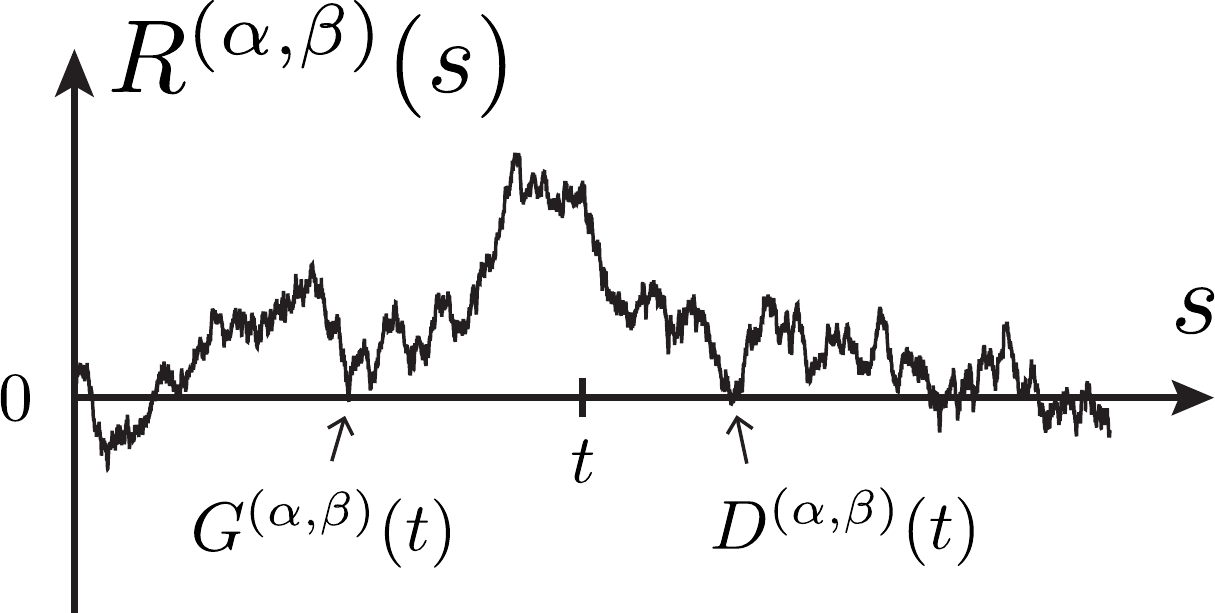}
\caption{One-dimensional Brownian motion $R^{(\alpha,\beta)}$ ($d=2$, $\alpha=\beta_1=\beta_2=1/2$)}
\label{fig:G-D}	
\end{figure}

For $t\geq0$, we will denote by $G^{(\alpha,\beta)}(t)$ the last zero of $R^{(\alpha,\beta)}$ before the time $t$ and by $D^{(\alpha,\beta)}(t)$ the first zero of $R^{(\alpha,\beta)}$ after the time $t$, that is,
\begin{align}
	G^{(\alpha,\beta)}(t)
	&:=
	\sup\{s\leq t:R^{(\alpha,\beta)}(s)=0\}, 
	\quad t\geq0,
	\\
	D^{(\alpha,\beta)}(t)
	&:=
	\:\inf\{s>t: R^{(\alpha,\beta)}(s)=0\},
	\quad t\geq0.
\end{align}
See Figure \ref{fig:G-D}.
The waiting times $(G^{(\alpha,\beta)}(t),D^{(\alpha,\beta)}(t))$ and related random variables are well-investigated in the studies of the last exit decomposition and diffusion bridges. See for example \cite{BPY, S97, Y06, SVY} and the references therein.
Let us recall a joint-distributional result for the occupation and waiting times.

\begin{Thm}[{\cite[Theorems 2 and 4]{BPY}}, see also {\cite[Subsection 1.3]{SVY}}]\label{thm:BPY-2}
For $t>0$, three joints 
\begin{gather*}
(G^{(\alpha,\beta)}(t), D^{(\alpha,\beta)}(t)),
\quad
\bigg(\bigg(\frac{Z_j^{(\alpha,\beta)}(G^{(\alpha,\beta)}(t))}{G^{(\alpha,\beta)}(t)}\bigg)_{j=1}^d, 
      \frac{L^{(\alpha,\beta)}(t)}
           {(G^{(\alpha,\beta)}(t))^\alpha} \bigg),
\\
\bigg(\frac{Z_j^{(\alpha,\beta)}(t)
             -Z_j^{(\alpha,\beta)}(G^{(\alpha,\beta)}(t))}
           {t-G^{(\alpha,\beta)}(t)}\bigg)_{j=1}^d	
\end{gather*}
are mutually independent, and the joint-laws of them are characterized by the following:
\begin{enumerate}
	\item The joint-law of $(G^{(\alpha,\beta)}(t), D^{(\alpha,\beta)}(t))$ is given by	
\begin{align}
  \bP[G^{(\alpha,\beta)}(t)\in \rd u,\; D^{(\alpha,\beta)}(t)\in\rd v]
  &=
  \frac{\alpha \sin(\alpha\pi)}{\pi}
  \frac{\rd u \rd v}{u^{1-\alpha}(v-u)^{1+\alpha}},
  \quad
  0<u<t<v.	
\end{align}
\item  For any bounded measurable function $f:\bR^{d+1}\to\bR$, it holds that
\begin{align}
  &\bE\Bigg[f\bigg(\bigg(\frac{Z_j^{(\alpha,\beta)}(G^{(\alpha,\beta)}(t))}{G^{(\alpha,\beta)}(t)}\bigg)_{j=1}^d, 
      \frac{L^{(\alpha,\beta)}(t)}
           {(G^{(\alpha,\beta)}(t))^\alpha} \bigg)\Bigg]
  \notag\\
  &=
  \bE\Bigg[
   f\bigg( \bigg(\frac{\xi_j}{\sum_{i=1}^d \xi_i}\bigg)_{j=1}^d, \frac{1}{(\sum_{i=1}^d \xi_i)^{\alpha}}\bigg)
   \frac{\Gamma(1+\alpha)}
          {(\sum_{i=1}^d \xi_i)^{\alpha}}\Bigg],
          \label{def:UV}
\end{align}
where $\xi_1,\dots,\xi_d$ denote independent $[0,\infty)$-valued $\alpha$-stable random variables characterized by \eqref{xi-lap}.
\item
\begin{align}
\bP\bigg[\bigg(\frac{Z_j^{(\alpha,\beta)}(t)
             -Z_j^{(\alpha,\beta)}(G^{(\alpha,\beta)}(t))}
           {t-G^{(\alpha,\beta)}(t)}\bigg)_{j=1}^d\in\cdot\bigg]
           =
           \sum_{j=1}^d \beta_j \delta_{(\mathbbm{1}_{\{i=j\}})_{i=1}^d}(\cdot),
\end{align}
where $\delta_x$ denotes the Dirac measure at $x$.	
\end{enumerate}

\end{Thm}
The joint-law of $(G^{(\alpha,\beta)}(t), D^{(\alpha,\beta)}(t))$ also appears as the limit distribution of waiting times in the renewal theory \cite{Dyn61, La-b}. Note that the one-dimensional law of $G^{(\alpha,\beta)}(t)/t\dist t/D^{(\alpha,\beta)}(t)$ is $\text{Beta}(\alpha,1-\alpha)$-distribution, which is the arcsine distribution in the case of $\alpha=1/2$. 
In the special case of $\alpha=\beta_1=1/2$, we know that the one-dimensional law of $Z^{(\alpha,\beta)}_1(G^{(\alpha,\beta)}(t))/G^{(\alpha,\beta)}(t)$ is uniform distribution on $[0,1]$. This fact was due to  L\'evy \cite{Le} and is called the uniform law.
In the case of $\alpha=1/2$ and $\beta_1=p\in(0,1)$, we have
\begin{align}
	\bP\bigg[\frac{Z^{(\alpha,\beta)}_1(G^{(\alpha,\beta)}(t))}{G^{(\alpha,\beta)}(t)}\in \rd x\bigg]
	=
	\frac{p(1-p)}{2}((1-2p)x+p^2)^{-3/2}\rd x,
	\quad 0<x<1.
\end{align}
In the case of $\alpha\in(0,1)$ and $\beta_1=p\in(0,1)$, we have
\begin{align}
	&\bP\bigg[\frac{Z^{(\alpha,\beta)}_1(G^{(\alpha,\beta)}(t))}{G^{(\alpha,\beta)}(t)}\leq x\bigg]
	\notag\\
	&=
	\frac{\sin(\alpha \pi)}{\pi}\int_0^x \frac{(1-p)(x-s)^{\alpha-1}s^\alpha \rd s}
	{p^2(1-s)^{2\alpha}+(1-p)^2s^{2\alpha}+2p(1-p)s^{\alpha}(1-s)^{\alpha}\cos(\alpha\pi)},
	\quad 0\leq x\leq 1.
\end{align}
We refer the reader to \cite[Example 1 and Theorem 3.1]{Y06} for the details.

\section{Main result}\label{sec:main}
\subsection{Functional limit theorem in the general setting}\label{subsec:functional}

Let us return to the setting introduced in Section \ref{sec:notation}. Set 
\begin{align}
	b_n
	:= 
	\frac{1}{\Gamma(1-\alpha)\mu_Y[\varphi \geq n]}
	\underset{n\to\infty}{\sim}
	\frac{\mu(Y)}{\Gamma(2-\alpha)}\frac{n}{w(n)}\in\cR_\alpha(\infty).
\end{align}
Note that $b_n\underset{n\to\infty}{\sim} \Gamma(1+\alpha)\mu(Y)a_n$.
We can now formulate our main result in the general setting.  

\begin{Thm}[Functional convergence of occupation and waiting time processes]\label{main1}
Let $T$ be a CEMPT on $(X,\cB,\mu)$ and
suppose that Assumptions \ref{ass:d-ray}, \ref{ass:reg-var} and \ref{ass:mixing} hold. 
Then it holds that 
\begin{align}
    &\bigg(\bigg(\frac{1}{n}S_{A_j}(nt)\bigg)_{j=1}^d,\;
          \frac{1}{b_n}S_Y(nt),\; \frac{1}{n}G_Y(nt),\; \frac{1}{n}D_Y(nt):t\geq0\bigg) 
    \notag\\
    \label{main1-1}
    &
    \underset{n\to\infty}{\overset{\cL(\mu)}{\Longrightarrow}}
    \Big(  \big(Z_j^{(\alpha,\beta)}(t)\big)_{j=1}^d,
     \;L^{(\alpha,\beta)}(t),\; G^{(\alpha,\beta)}(t), \;D^{(\alpha,\beta)}(t) :t\geq0\Big),
    \quad\text{in $D([0,\infty),\bR^{d+3})$}.
\end{align}
\end{Thm}

The proof of Theorem \ref{main1} will be given in Section \ref{sec:proof}. We immediately obtain the following corollary.

\begin{Cor}[Marginal convergence]\label{cor:main1}
Let $T$ be a CEMPT on $(X,\cB,\mu)$ and suppose that Assumptions \ref{ass:d-ray}, \ref{ass:reg-var} and \ref{ass:mixing} hold. 
Then, 
\begin{align}
    &\bigg(\bigg(\frac{1}{n}S_{A_j}(n)\bigg)_{j=1}^d,\;\bigg(\frac{1}{n}S_{A_j}(G_Y(n))\bigg)_{j=1}^d, \;
          \frac{1}{b_n}S_Y(n),\; \frac{1}{n}G_Y(n),\;\frac{1}{n}D_Y(n)\bigg)    
        \notag  \\
    &\underset{n\to\infty}{\overset{\cL(\mu)}{\Longrightarrow}}
    \Big(  \big(Z_j^{(\alpha,\beta)}(1)\big)_{j=1}^d, \;\big(Z_j^{(\alpha,\beta)}(G^{(\alpha,\beta)}(1))\big)_{j=1}^d,
     \;L^{(\alpha,\beta)}(1),\; G^{(\alpha,\beta)}(1), \;D^{(\alpha,\beta)}(1) \bigg),
    \quad\text{in $\bR^{2d+3}$}.
    \label{cor:main1-1}
\end{align}	
\end{Cor}
Recall that Theorems \ref{thm:BPY} and \ref{thm:BPY-2} characterizes the joint-laws of the limit random variables in \eqref{cor:main1-1}.

\begin{Rem}\label{rem:sym-diff}
In addition, suppose that Borel subsets $A_1',\dots,A_d',Y'\in\cB([0,1])$ satisfy
\begin{align*}
	\mu(A_j\triangle A_j' )<\infty, \quad j=1,\dots,d,
	\quad\text{and}\quad
	\mu(Y')\in(0,\infty),
\end{align*}
where $A\triangle B=(A\setminus B)\cup (B \setminus A)$ denotes the symmetric difference of	$A$ and $B$. 
As we shall see in Remark \ref{rem:main-modif-proof}, the convergence \eqref{main1} is equivalent to the following:
\begin{align}
    &\bigg(\bigg(\frac{1}{n}S_{A_j'}(nt)\bigg)_{j=1}^d,\;
          \frac{1}{b_n}S_{Y'}(nt),\; \frac{1}{n}G_{Y'}(nt), \;\frac{1}{n}D_{Y'}(nt):t\geq0\bigg) 
    \notag\\
    &
    \underset{n\to\infty}{\overset{\cL(\mu)}{\Longrightarrow}}
    \Big(  \big(Z_j^{(\alpha,\beta)}(t)\big)_{j=1}^d,\;
     \frac{\mu(Y')}{\mu(Y)}L^{(\alpha,\beta)}(t),\;
     G^{(\alpha,\beta)}(t), \;D^{(\alpha,\beta)}(t)
      :t\geq0\Big),
    \quad\text{in $D([0,\infty),\bR^{d+3})$}.
    \label{rem:sym-diff-1}	
\end{align}
\end{Rem}

\begin{Rem}[Application to null-recurrent Markov chains on multiray]\label{rem:markov}
As shown in \cite[Subsection 2.5]{SY}, Assumptions \ref{ass:d-ray} and \ref{ass:reg-var} are satisfied in the setting of some class of null-recurrent Markov chains on multiray. In this setting, Assumption \ref{ass:mixing} is also satisfied, since the excursion processes of recurrent Markov chains away from a point are i.i.d. See for example \cite[Proposition 8.15]{Kal02}.	Hence Theorem \ref{main1} can be applied to these Markov chains.
\end{Rem}

\subsection{Application to interval maps with indifferent fixed points}
\label{subsec:int}
 
Let $d\geq2$ be a positive integer and $0=a_0=x_1<a_1<x_2<\dots<x_d=a_d=1$.
Set $J_1:=[a_0,a_1), J_2:=[a_1,a_2),\dots,J_d:=[a_{d-1},a_d]$.
Suppose that an interval map $T:[0,1]\to[0,1]$ satisfies the following two conditions: for each $j=1,\dots,d$,
\begin{itemize}
\item[(1)] The restriction $T|_{J_j}$ over $J_j$ can be extended to a $C^2$-bijective map $T_j:\bar{J_j}\to[0,1]$;
\item[(2)]
   $T_j x_j=x_j$, $T_j'x_j=1$ and $(x-x_j)T_j''x>0$ for any $x\in \bar{J_j}\setminus \{x_j\}$.
\end{itemize}
\begin{figure}[h]
\center
\includegraphics[height=5cm]{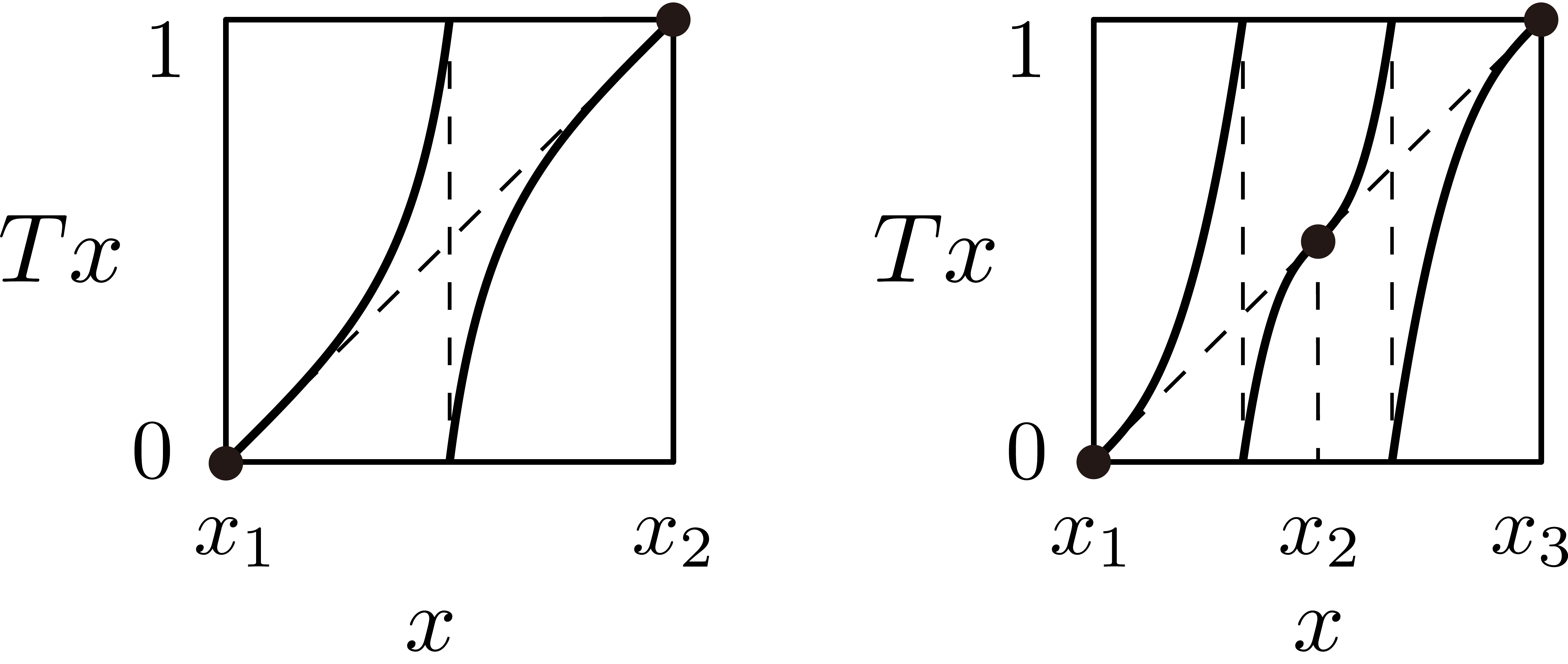}
\caption{Graphs of $T$ in the cases of $d=2$ and of $d=3$}	
\end{figure}
Therefore we have $T_j' x>1$ on $\bar{J_j}\setminus\{x_j\}$ for $j=1,\dots,d$.
Each $x_j$ is called an \emph{indifferent fixed point} and a \emph{regular source}. In this case, $T$ has a unique (up to scalar multiplication) $\sigma$-finite invariant measure $\mu(\rd x)$ equivalent to the Lebesgue measure $\rd x$, and $T$ is a CEMPT on $\big([0,1],\cB([0,1]),\mu\big)$. Any neighborhood of $x_j$ has infinite volume with respect to $\mu$. More specifically,
for any $\varepsilon>0$, it holds that
\begin{align}\label{infinite-measure}
	\mu((x_j-\varepsilon,x_j+\varepsilon))=\infty,
	\quad j=1,\dots,d,
	\quad\text{and}\quad
	\mu\bigg([0,1]\setminus \bigcup_{j=1}^d (x_j-\varepsilon, x_j+\varepsilon)\bigg)<\infty.
\end{align}
The density $\mu(\rd x)/\rd x$ has the version $h(x)$ which is continuous and positive on $[0,1]\setminus \{x_j\}_{j=1}^d$.
For the details, we refer the reader to Thaler \cite{T80, T83}.

\begin{Ass}[Regular variation]\label{ass:indiff}
	There exist constants $\alpha\in(0,1)$ and $c=(c_1,\dots,c_d)\in(0,\infty]^{d}\setminus\{(\infty,\dots,\infty)\}$ and an increasing function $\Psi:(0,\infty)\to(0,\infty)$ such that $\Psi\in\cR_{1+1/\alpha}(0+)$ and
 \begin{align}\label{reg:T}
 	|Tx-x|\underset{x\to x_j}{\sim} c_j \Psi\big(|x-x_j|\big),
 	\quad
 	j=1,\dots,d.
\end{align}
\end{Ass}
We will assume Assumption \ref{ass:indiff} from now on. For $j=1,\dots,d$, let us denote by $f_j:[0,1]\to \bar{J_j}$ the inverse function of $T_j$.
Set 
\begin{align}
  v_j
	&:=
	\begin{cases}
	\sum_{i\neq j} (h\circ f_j) (x_i) f'_j(x_i), 
	&j=1,d,
	\\
	2\sum_{i\neq j} (h\circ f_j) (x_i) f'_j(x_i),
	&j=2,\dots,d-1,	
	\end{cases}
    \\
  \beta_j
  &:=\frac{c_j^{-\alpha}v_j}{\sum_{i=1}^d c_i^{-\alpha}v_i},
  \quad j=1,\dots,d,
  \label{beta_j}
  \quad
  \text{and}
  \quad
  \beta
  :=
  (\beta_1,\dots,\beta_d),
 \\
 \label{phi}
 b_n
  &:=\frac{1}{\Gamma(1-\alpha)\sum_{i=1}^d c_i^{-\alpha}v_i}\inf\bigg\{s\geq1: \frac{\alpha s^{-1}}{\Psi(s^{-1})}>n \bigg\},
  \quad n\geq0.	
\end{align}
Using basic theory of regular variation (see \cite[Theorem 1.5.12]{BGT}), we have $(b_n)_{n\geq0}\in\cR_\alpha(\infty)$.
Let $A_1,\dots,A_d,Y\in \cB([0,1])$ be disjoint Borel subsets of $[0,1]$ such that, for some $\varepsilon>0$,
\begin{align}
(x_j-\varepsilon,x_j+\varepsilon)\cap[0,1]\subset A_j,
\quad j=1,\dots,d,
\quad
\text{and}\quad Y = [0,1]\setminus \sum_{j=1}^d A_j.
\label{eq:AY}
\end{align}
For $A\in\cB([0,1])$ and $t\geq0$, we define $S_A(t)$, $G_A(t)$ and $D_A(t)$ in the same way as \eqref{sA}, \eqref{def:G_A} and \eqref{def:D_A}, respectively.

We can now formulate our main theorem in the setting of intermittent interval maps.

\begin{Thm}[Functional convergence of occupation time processes]\label{main2}
Suppose that Assumption \ref{ass:indiff} holds. 
 Then, it holds that
\begin{align}
    &\bigg(\bigg(\frac{1}{n}S_{A_j}(nt)\bigg)_{j=1}^d,\;
          \frac{1}{b_n}S_{Y}(nt),\; \frac{1}{n}G_Y(nt), \; \frac{1}{n} D_Y(nt):t\geq0\bigg) 
    \notag\\
    \label{main2-1}
    &
    \underset{n\to\infty}{\overset{\cL(\mu)}{\Longrightarrow}}
    \Big(  \big(Z_1^{(\alpha,\beta)}(t)\big)_{j=1}^d,\;
     \mu(Y) L^{(\alpha,\beta)}(t), \;G^{(\alpha,\beta)}(t),\;D^{(\alpha,\beta)}(t) :t\geq0\Big),
    \quad\text{in $D([0,\infty),\bR^{d+3})$}.
\end{align}	
\end{Thm}

The proof of Theorem \ref{main2} will be given in Section \ref{sec:proof2}. 

\begin{Ex}[Boole's transformation]\label{ex:boole}
Let us define the interval map $T$ by (\ref{boole}). This map satisfies our assumptions for $d=2$, $\alpha=1/2$, $c_1=c_2=1$ and $\Psi(s)=s^3$. It is easy to check that $\beta_1=\beta_2=1/2$ and $b_n=\sqrt{n/(2\pi)}$.  We use Theorem \ref{main2} to obtain the convergence (\ref{our-main}). Recall Remark \ref{rem:skew-bes2}.
\end{Ex}

\begin{Rem}
More precisely, we can also prove the following refinement of Theorem \ref{main2}. Let $A^+_{1}, A^-_{2},A^+_{2},\dots,A^-_{d},Y\in\cB([0,1])$ be disjoint Borel subsets of $[0,1]$ such that, for some $\varepsilon>0$,
\begin{align*}
	(x_j-\varepsilon,x_j)\subset A^-_{j}
	\quad\text{and}\quad
	(x_j,x_j+\varepsilon)\subset A^+_{j},
	\quad j=1,\dots,d,
\end{align*}
and $Y=[0,1]\setminus\sum_{j,\pm} A^\pm_{j}$.	Then we can formulate a functional limit theorem for
\begin{align*}
	\bigg(\bigg(\frac{1}{n}S_{A^\sigma_{j}}(nt)\bigg)_{\substack{\sigma=\pm\\ j=1,\dots,d}},\;
          \frac{1}{b_n}S_{Y}(nt),\;\frac{1}{n}G_Y(nt),\;\frac{1}{n}D_Y(nt):t\geq0\bigg),
\end{align*}
which is a functional and joint-distributional extension of \cite[Corollary 2.12]{SY}.
The statement and proof are almost the same as those of Theorem \ref{main2}, so we omit them. 
\end{Rem}

\section{Space of c\`adl\`ag functions and Skorokhod $J_1$-topology}
\label{sec:Skorokhod}
For the proof of Theorem \ref{main1}, 
we need several auxiliary results on the Skorokhod $J_1$-topology. For the basic discussion of the Skorokhod $J_1$-topology, we refer the reader to Bingham \cite[Section 2]{B71}, Ethier--Kurtz \cite[Chapter 3]{EK} and Jacod--Shiryaev \cite[Chapter VI]{JS}.

Let $d\geq1$ be a positive integer.
We will denote by $D([0,\infty),\bR^d)$ the space of $\bR^d$-valued c\`adl\`ag functions $x:[0,\infty)\to \bR^d$, that is,
\begin{align*}
	D([0,\infty),\bR^d)
	:=\{x=(x(t):t\geq0) :x(t+)=x(t)\in\bR^d, \;x(t-) \in \bR^d\;\; \text{for any $t\geq0$}\}.
\end{align*}    
Here, for $t\geq0$, 
\begin{align*}
   x(t+):=\lim_{s\downarrow t}x(s) 
   \quad
   \text{and}
   \quad   
    x(t-):=\begin{cases}\displaystyle
	         \lim_{s\uparrow t}x(s), & t>0,
	         \\
	         x(0), & t=0.
             \end{cases}
\end{align*}
Let $\Lambda'$ be the collection of strictly increasing functions $\lambda$ mapping $[0,\infty)$ onto itself, and let $\Lambda\subset \Lambda'$ be the collection of Lipschitz continuous functions $\lambda\in\Lambda'$ such that
\begin{align}\label{def:log-time}
\gamma(\lambda):=\sup_{0\leq s<t <\infty} 
                  \bigg|\log\frac{\lambda(t)-\lambda(s)}{t-s}\bigg|
               <\infty.	
\end{align}
Set $a\vee b=\max\{a,b\}$ and $a\wedge b=\min\{a,b\}$.
For $x, y \in D([0,\infty),\bR^d)$, $\lambda\in\Lambda$ and $u\geq0$, set
\begin{align}\label{def:pre-met}
	\rho(x,y,\lambda,u)
	:=
	1\wedge\sup_{t\geq0} \big|x(t\wedge u )-y(\lambda(t)\wedge u)\big|,
\end{align}
Let us define the metric $\rho(x,y)$ on $D([0,\infty),\bR^d)$ by
\begin{align}\label{def:metric}
	\rho(x,y):=\inf_{\lambda\in\Lambda}
	        \bigg[\;\gamma(\lambda)\vee
	              \int_0^\infty \rho(x,y,\lambda,u)e^{-u} \rd u\bigg].
\end{align}
Then the metric space $\big(D([0,\infty),\bR^d),\rho\big)$ is complete and separable. The Polish topology of $D([0,\infty),\bR^d)$ generated by $\rho$ is called the \emph{Skorokhod $J_1$-topology}.
Let $(x_n)_{n\geq1}\subset D([0,\infty),\bR^d)$ be a sequence and $x_\infty\in D([0,\infty),\bR^d)$. Then the $x_n$ converge to $x_\infty$ with respect to the Skorokhod $J_1$-topology, i.e., $\rho(x_n, x_\infty)\to0$, if and only if, there exists a sequence $(\lambda_n)_{n\geq1}\subset\Lambda'$ such that, for any $t_0>0$,
\begin{align}
	\sup_{0\leq t\leq t_0}|\lambda_n(t)-t|\underset{n\to\infty}{\to}0
	\quad\text{and}\quad
	\sup_{0\leq t\leq t_0}\big|x_n(\lambda_n(t))-x_\infty(t)\big|\underset{n\to\infty}{\to}0.
\label{conv:lambda}
\end{align}

Let $D_0\subset D([0,\infty),\bR)$ be the space of non-decreasing c\`adl\`ag functions $x:[0,\infty)\to[0,\infty)$ with $\lim_{t\to\infty}x(t)=\infty$. 
For $x=(x(t):t\geq0)\in D_0$, let  $x^{-1}=(x^{-1}(s):s\geq0)$ denotes the right-continuous inverse of $x$, that is,
\begin{align}
x^{-1}(s):=\inf\{t>0: x(t)>s\},
\quad s\geq0.
\label{def:r-inverse}
\end{align}

\begin{Lem}
For $x\in D_0$, let us define non-decreasing functions $(G(t):t\geq0)$ and $(D(t):t\geq0)$ by
\begin{align}
   G(t)&:=	 \sup\{x(s):x(s)\leq t\}, \quad t\geq0,
   \\
   D(t)&:= \inf\{x(s):x(s)>t\}, \quad t\geq0,	
\end{align}
where $\sup \emptyset=0$. Then $G(t+)=G(t)$ and $D(t+)=D(t)$. Hence $(G(t):t\geq0), (D(t):t\geq0)\in D_0$. 
\end{Lem}

\begin{proof}
We proceed by cases. 
\begin{enumerate}
\item
Assume $t<D(t)$. Then, for $t'\in(t,D(t))$, we have $D(t')=D(t)$ by the definition. Hence $D(t+)=D(t)$. 
\item
If $t=D(t)$, then there exists a strictly decreasing sequence $(s_n)_{n\geq1}$ such that $x(s_n)\downarrow t$, and hence $x(s_n)\leq D(x({s_n}))\leq x(s_{n-1})$. This implies that $D(t+)=\lim_{n\to\infty} D(x(s_n))=t=D(t)$. We also remark that $x(\lim_{n\to\infty}s_n)=t$ by the right-continuity of $x$, and hence  $G(t)=t$.

\item
Assume $G(t)<t$. By the contraposition of the above discussion, we have $t<D(t)$. For $t'\in (t,D(t))$, we have $G(t')=G(t)$ by the definition. Therefore $G(t+)=G(t)$.
\item
Assume $G(t)=t$. Let $(t_n)_{n\geq1}$ be a sequence with $t_n\downarrow t$, as $n\to\infty$. Then we have $t \leq G(t_n) \leq t_n$ and hence $G(t+)=\lim_{n\to\infty}G(t_n)=t=G(t)$.
\end{enumerate}
\end{proof}

\begin{Lem}\label{lem:conv-G-D}
For $n=1,\dots,\infty$, let $x_n\in D_0$.
Set
\begin{align}
   G_n(t)&:=	 \sup\{x_n(s):x_n(s)\leq t\}, \quad t\geq0,
   \\
   D_n(t)&:= \inf\{x_n(s):x_n(s)>t\}, \quad t\geq0,
\end{align}
Assume that $x_n\to x_\infty$, in $D([0,\infty),\bR)$. Then, it holds that 
\begin{align}
	G_n(t)\underset{n\to\infty}{\to} G_\infty(t), \quad \text{in $\bR$, for any continuity points $t\geq0$ of $G_\infty$},
	\label{conv:G}
	\\
	D_n(t)\underset{n\to\infty}{\to} D_\infty(t) \quad \text{in $\bR$, for any continuity points $t\geq0$ of $D_\infty$}.
	\label{conv:D}
\end{align}
\end{Lem}

\begin{Rem}
Let us give an example of $G_n(t)\not\to G_\infty(t)$ at a discontinuity point $t$ of $G_\infty$. Set
\begin{align*}
  x_n(s):= (1+n^{-1})\mathbbm{1}_{\{s\geq1\}},
  \quad n=1,\dots,\infty,\;s\geq0,	
\end{align*}
where $\infty^{-1}=0$. Then $x_n\to x_\infty$, in $D([0,\infty),\bR)$. 
Note that $0=G_\infty(1-)<G_\infty(1)=1$. In this case,
we have $0=G_n(1)\not\to G_\infty(1)$.
\end{Rem}

\begin{proof}[Proof of Lemma \ref{lem:conv-G-D}]
Let us prove \eqref{conv:G}.
In the cases of $t=0$ and of $t<x_\infty(0)$, it is obvious that $G_n(t)\to G_\infty(t)=0$. Let us consider the other cases. Let us take $(\lambda_n)\subset\Lambda'$ so that \eqref{conv:lambda} holds. We proceed by cases.
\begin{enumerate}
\item Assume $x_\infty(0)\leq G_\infty(t-)=G_\infty(t)<t$. Set $s_0:=x_\infty^{-1}(t)>0$. Then 
\begin{align}\label{conv:G-1}
G_\infty(t)=x_\infty(s_0-)<t<x_\infty(s_0).
\end{align}	
This implies
\begin{align}\label{conv:G-2}
  x_n(\lambda_n(s_0)-)\underset{n\to\infty}{\to} x_\infty(s_0-)<t
  \quad\text{and}\quad
  x_n(\lambda_n(s_0)) \underset{n\to\infty}{\to} x_\infty(s_0)>t.	
\end{align}
Hence $G_n(t)=x_n(\lambda_n(s_0)-)$ for sufficiently large $n$.
Therefore $G_n(t)\to G_\infty(t)$.
\item Assume $G_\infty(t-)=G_\infty(t)=t>0$. Then there exists a strictly increasing sequence $(s_m)_{m\geq1}$ such that $x_\infty(s_m)\uparrow t$. Then
\begin{align}
	\liminf_{n\to\infty} G_n(t)\geq \lim_{n\to\infty}x_n(\lambda_n(s_m))= x_\infty(s_m)\underset{m\to\infty}{\to}t.
\end{align}
Since $G_n(t)\leq t$, we obtain $G_n(t)\to t=G_\infty(t)$. 
\end{enumerate}
The proof of \eqref{conv:D} is almost the same. So we omit it.
%
%
\end{proof}

Let $X=(X(t):t\geq0)$ be a $D([0,\infty),\bR^d)$-valued random variable, that is, a stochastic process whose path-function lies in $D([0,\infty),\bR^d)$. We say that $X$ is stochastically continuous if $\bP[X(t-)=X(t)]=1$, for any $t\geq0$. 
\begin{Prop}[Bingham {\cite[Theorem 3]{B71}}]\label{prop:finite-dim}
Let $d\geq1$ be a positive integer. For each $n=1,2,\dots,\infty$ and $j=1,\dots,d$, let $X_{n,j}=(X_{n,j}(t):t\geq0)$ be a $D_0$-valued random variable. Assume that
\begin{itemize}

\item[\rm (i)]
$X_{\infty,1},\dots,X_{\infty,d}$ are stochastically continuous.

\item[\rm (ii)]
The finite-dimensional marginal laws of $(X_{n,j}:j=1,\dots,d)$ converge as $n\to\infty$ to those of $(X_{\infty,j}:j=1,\dots,d)$. 
\end{itemize}
Then, it holds that
\begin{align*}
	(X_{n,j}:j=1,\dots,d)\underset{n\to\infty}{\dto}(X_{\infty,j}:j=1,\dots,d),
	\quad
	\text{in $D([0,\infty),\bR^d)$.}
\end{align*}
\end{Prop}

The following lemma is one of the keys for the strong distributional convergence of waiting times, as we shall see in Section \ref{sec:proof}.

\begin{Lem}\label{lem:dist-conv-G-D}
For $n=1,\dots,\infty$, let $X_n=(X_n(t):t\geq0)$ be a $D_0$-valued random variable. Set	
\begin{align}
G_n(t)
	&:=
	\sup\{X_{n}(s):X_{n}(s)\leq t\},\quad t\geq0,
	\\
	D_n(t)
	&:=
	\inf\{X_{n}(s): X_{n}(s)>t\}, \quad t\geq0.
\end{align}
Assume that 
\begin{itemize}
\item[\rm (i)] $G_\infty$ and $D_\infty$ are stochastically continuous.
\item[\rm (ii)] $X_n\underset{n\to\infty}{\dto} X_\infty$, in $D([0,\infty),\bR)$.
\end{itemize}
 Then it holds that
 \begin{align}
 (G_n,D_n)\underset{n\to\infty}{\dto}(G_\infty,D_\infty),
 \quad \text{in $D([0,\infty),\bR^2)$}.	
 \end{align}
\end{Lem}

\begin{Rem}
For example, suppose that $X_\infty$ is a stable subordinator. Then for each $t>0$, it holds that 
\begin{align*}
   \bP[G_\infty(t-)=G_\infty(t), D_\infty(t-)=D_\infty(t)]
   \geq
   \bP\Big[t\notin \bar{\{X_\infty(s):s\geq0\}}\Big]=1.	
\end{align*}
Therefore the condition (i) of Lemma \ref{lem:dist-conv-G-D} is satisfied.	
\end{Rem}

\begin{proof}[Proof of Lemma \ref{lem:dist-conv-G-D}]
 By the Skorokhod coupling (see for example \cite[Theorem 4.30]{Kal02}), there is no loss of generality in assuming 
\begin{align}
	X_n \underset{n\to\infty}{\to}
	X_\infty,
	\quad \text{in $D([0,\infty),\bR)$, a.s.}
\end{align}
Then the condition (i) and Lemma \ref{lem:conv-G-D} implies that, for each $t\geq0$,
\begin{align}
  (G_n(t),D_n(t))
  \underset{n\to\infty}{\to}
  (G_\infty(t),D_\infty(t)),
  \quad
  \text{in $\bR^2$, a.s.}	
\end{align}
Hence the finite-dimensional marginal laws of $(G_n, D_n)$ converge as $n\to\infty$ to those of $(G_\infty,D_\infty)$. Therefore we use Proposition \ref{prop:finite-dim} to obtain the desired result.
\end{proof}

\begin{Lem}[Fujihara--Kawamura--Yano {\cite[Lemma 2.3]{FKY}}]\label{lem:cad-inv}
For $n=1,\dots,\infty$, let $x_n,y_n\in D_0$. 
	 Let us fix $t\geq0$. Assume that the following conditions are satisfied:
\begin{itemize}
   \item[\rm (i)]	
   $y_\infty(x^{-1}_\infty(t)-)=y_\infty(x^{-1}_\infty(t))$ and  $\begin{cases} x_\infty(0)=x^{-1}_\infty(0)=0,
 &\text{if $t=0$,}
 \\
	x^{-1}_\infty(t-)=x^{-1}_\infty(t),
	&\text{if $t>0$.}
\end{cases}
$
   \item[\rm (ii)] 
   $(x_n,y_n)\underset{n\to\infty}{\to} (x_\infty,y_\infty)$, in $D([0,\infty),\bR^2)$.
\end{itemize}
Then, it holds that
\begin{align*}
   y_n\big(x_n^{-1}(t)\big) 
   \underset{n\to\infty}{\to}
    y_\infty\big( x_\infty^{-1}(t)\big),
    \quad
    \text{in $\bR$}.
\end{align*}
\end{Lem}

The following lemma is one of the keys for the strong distributional convergence of the occupation time processes. The proof is almost the same as that of Lemma \ref{lem:dist-conv-G-D}, so we omit it.

\begin{Lem}\label{lem:inverse}
Let $d\geq1$ be a positive integer. For each $n=1,2,\dots,\infty$ and $j=1,\dots,d$, let $X_{n,j}=(X_{n,j}(t):t\geq0)$ be a $D_0$-valued random variable. 
Assume that
\begin{itemize}
\item[\rm (i)]  $X_{\infty,j}(0)=X_{\infty,j}^{-1}(0)=0$, for $j=1,\dots,d$.

\item[\rm (ii)] $X^{-1}_{\infty,1},\dots,X^{-1}_{\infty,d}$ are stochastically continuous.

\item[\rm (iii)] $\bP[X_i(X^{-1}_j(t)-)=X_i(X^{-1}_j(t))]=1$, for any $t\geq0$ and for any distinct $i,j$.

\item[\rm (iv)]	$(X_{n,j}:j=1,\dots,d)
       \underset{n\to\infty}{\dto}
       (X_{\infty,j}:j=1,\dots,d)$, in $D([0,\infty),\bR^{d})$.
\end{itemize}
Then, it holds that	
\begin{align}\label{lem:inverse-1}
	\Big(X_{n,i}\big(X_{n,j}^{-1}(t)\big):t\geq0,\; i\neq j\Big)
       \underset{n\to\infty}{\dto}
       \Big(X_{\infty,i}\big(X_{\infty,j}^{-1}(t)\big):t\geq0,\; i\neq j\Big), 
       \notag\\
       \text{in $D([0,\infty),\bR^{d(d-1)})$}.
\end{align}
\end{Lem}

\begin{Rem}
For example, suppose that $X_{\infty,1},\dots, X_{\infty,d}$ are independent stable subordinators. In this case, we can easily check that the conditions (i), (ii) and (iii) are satisfied.	
\end{Rem}

\begin{Cor}\label{cor:inverse}
	For $n=1,\dots,\infty$, let $X_n$ be a $D_0$-valued random variable. Assume that $X_{\infty}(0)=X_{\infty}^{-1}(0)=0$, $X_\infty$ and $X^{-1}_{\infty}$ are stochastically continuous, and $X_n\underset{n\to\infty}{\dto}X_\infty$, in $D([0,\infty),\bR)$.
Then, it holds that	$(X_n, X^{-1}_n)\underset{n\to\infty}{\dto} (X_\infty, X^{-1}_\infty)$, in $D([0,\infty),\bR^2)$.
\end{Cor}

\section{Proof of Theorem \ref{main1}}\label{sec:proof}

For the proof of Theorem \ref{main1}, we imitate the methods of Barlow--Pitman--Yor \cite{BPY}, Watanabe \cite{W95}, Fujihara--Kawamura--Yano \cite{FKY} and Yano \cite{Y17}. We will represent waiting times and occupation times in terms of excursion lengths. Combining these representations with a functional convergence of excursion lengths, we will obtain the desired result.

\subsection{Representation formulae}

We now return to the setting introduced in Section \ref{sec:notation}. 
Let us denote by $\varphi(t)=\varphi(t)(x)$ the $\lfloor t +1\rfloor$th return time of the orbit $(T^k x)_{k\geq0}$ for $Y$:
\begin{align}
  \varphi(t)
  &:=\sum_{k=0}^{\lfloor t \rfloor}\varphi\circ T_Y^k= S_Y^{-1}(t)=\min\{u\geq0:S_Y(u)=\lfloor t+1 \rfloor\},
  \quad t\geq0.
\end{align}
Recall that $S_Y^{-1}$ denotes the right-continuous inverse of $S_Y$. See \eqref{def:r-inverse}.   It is obvious that
\begin{align}
	\{k\geq1:T^k x \in Y\}
	=
	\{\varphi(t)(x):t\geq0\},
	\quad x\in X.
\end{align}
By the definition of $G_Y$ and $D_Y$, we immediately obtain the following representation formulae:

\begin{Lem}\label{lem:discrete-represent-G-D}
	Suppose that $T$ is a CEMPT on $(X,\cB,\mu)$ and Assumption \ref{ass:d-ray} holds. Then it holds that
\begin{align}
   G_Y(t)&=\sup\{\varphi(s):\varphi(s)\leq t\}, \quad t\geq0,
   \\
   D_Y(t)&=\inf\{\varphi(s):\varphi(s)> t\},
   \quad t\geq0.	
\end{align}
\end{Lem}

In addition, set
\begin{align}
	\eta_j(t)
	:=
	S_{A_j}(\varphi(t))
	=
	\sum_{k=0}^{\lfloor t \rfloor} \ell_j\circ T^{k}_Y,
	\quad
	t\geq0, \;j=1,\dots,d,
\end{align}
that is, $\eta_j(t)=\eta_j(t)(x)$ denotes the amount of time which the orbit spends on $A_j$ up to time $\varphi(t)=\varphi(t)(x)$. Since $T$ is a CEMPT, we have $(\eta_j(t):t\geq0)\in D_0$, $\mu_Y$-a.e.

\begin{Lem}[Discrete Williams formulae]\label{lem-dW}
Suppose that $T$ is a CEMPT on $(X,\cB,\mu)$ and Assumption \ref{ass:d-ray} holds. Then it holds that
\begin{align}\label{d-W}
  S^{-1}_{A_j}(t)
  &=
  \lfloor t+1\rfloor+\sum_{\substack{i=1,\dots,d,\\ i\neq j}}
  \eta_i\big(\eta_j^{-1}(t)\big)
              +\eta_j^{-1}(t),
  \quad t\geq0,\;j=1,\dots,d,
  \\
  \label{d-W0}
  \varphi(t)&=\lfloor t+1\rfloor +\sum_{i=1}^d \eta_i(t),
  \quad t\geq0.	
\end{align}
\end{Lem}
\begin{proof}
Set $N:=\eta_j^{-1}(t)$, which is a non-negative integer, $\mu$-a.e. By the definition of $\eta_j$, we have
\begin{align}
S_{A_j}(\varphi(N-1))\leq t < S_{A_j}(\varphi(N)), 
\end{align}
that is,
\begin{align}\label{eq:11}
\varphi(N-1) \leq S_{A_j}^{-1}(t) < \varphi(N), 	
\end{align}
where it is understood that $\varphi(-1)=0$.
Combining \eqref{eq:11} with Assumption \ref{ass:d-ray}, we have
\begin{align}\label{d-Wa}
	S_{A_i} \big(S^{-1}_{A_j}(t)\big)=\eta_i(N), \quad i\neq j,
\end{align}
and
\begin{align}\label{d-Wb}
	S_Y\big(S^{-1}_{A_j}(t)\big)=N.
\end{align}
It also follows immediately that 
\begin{align}\label{d-Wc}
	S_{A_j}(S_{A_j}^{-1}(t))=\inf\{S_{A_j}(u):S_{A_j}(u)>t\}=\lfloor t+1\rfloor.
\end{align}
In addition, we have 
\begin{align}\label{d-W3}
	u =\sum_{i=1}^d S_{A_i}(u)+S_Y(u),
	\quad
	u\geq0.
\end{align}
Substituting $u=S^{-1}_{A_j}(t)$ in \eqref{d-W3} and using \eqref{d-Wa}, \eqref{d-Wb} and \eqref{d-Wc},
 we obtain \eqref{d-W}. We can easily obtain \eqref{d-W0} by the definition. 
\end{proof}

We now return to the setting introduced in Section \ref{sec:bessel}.
Recall that the zero set of $R^{(\alpha,\beta)}$ coincides with the closure of the image set of the inverse local time $\eta^{(\alpha,\beta)}$:
\begin{align}
	\{t\geq0:R^{(\alpha,\beta)}(t)=0\}
	=
	\bar{\{\eta^{(\alpha,\beta)}(s):s\geq0\}}.
\end{align}
Hence we immediately obtain the following lemma:

\begin{Lem}\label{lem:represent-G-D}
It holds that
\begin{align*}
	G^{(\alpha,\beta)}(t)
	&=
	\sup\{\eta^{(\alpha,\beta)}(s):\eta^{(\alpha,\beta)}(s)\leq t\}, \quad t\geq0,
	\\
	D^{(\alpha,\beta)}(t)
	&=\inf\{\eta^{(\alpha,\beta)}(s):\eta^{(\alpha,\beta)}(s)> t\}, \quad t\geq0.
\end{align*}
\end{Lem}

Set
\begin{align}
    \eta_j^{(\alpha,\beta)}(s)
    :=
    Z^{(\alpha,\beta)}_j\big(\eta^{(\alpha,\beta)}(s)\big)
    =
    \sum_{u\leq s} \inf\{t>0: e_u^{(\alpha,\beta)}(t)=0\}\mathbbm{1}_{\{e^{(\alpha,\beta)}_u\in C([0,\infty),I_j)\}},
    \notag\\
    s\geq0,\;
    j=1,\dots,d.	
\end{align}
By the It\^o excursion theory, we know that $\eta^{(\alpha,\beta)}_1,\dots,\eta^{(\alpha,\beta)}_d$ are independent $\alpha$-stable subordinators with Laplace transforms
\begin{align}
  \bE[\exp(-\lambda \eta_j^{(\alpha,\beta)}(s))]
  =
  \exp(-\lambda^{\alpha}\beta_j s),
  \quad
  t\geq0,
  \;j=1,\dots,d.	
\end{align}

\begin{Lem}[Williams formulae]\label{lem:Wil}
It holds that
\begin{align}
	\big(Z_j^{(\alpha,\beta)}\big)^{-1}(t)
	&=
	t+\sum_{\substack{i=1,\dots,d,\\ i\neq j}} \eta^{(\alpha,\beta)}_i
	\Big(\big(\eta^{(\alpha,\beta)}_j\big)^{-1}(t)\Big),
	\quad
	t\geq0,\;j=1,\dots,d.
	\label{lem:Wil0}
 \\
  \eta^{(\alpha,\beta)}(t)
  &=
  \sum_{i=1}^d \eta^{(\alpha,\beta)}_i(t),
  \quad
  t\geq0.	
\end{align}
\end{Lem}

 For the proof, we refer the reader to Yano \cite[Theorem 3.1]{Y17}. See also Watanabe \cite[Proposition 1]{W95}.
The prototype of the equality \eqref{lem:Wil0} was obtained by Williams \cite[Theorem 1]{Wi69}.

\subsection{Functional convergence of partial sum of excursion lengths}

We will recall a sufficient condition for strong distributional convergence:

\begin{Prop}[Zweim\"uller {\cite[Theorem 1]{Z07b}}]\label{lem:st-dist}
Let $T$ be a CEMPT on $(X, \cB, \mu)$, let $\nu_0\ll \mu$ be a probability measure on $X$, and let $(U,\rho)$ be a separable metric space.  
Assume that measurable functions $F_n: X \to U$ $(n\in\bN)$ satisfy the following:
\begin{enumerate}
\item[\rm (i)]
 $F_n \underset{n\to\infty}{\overset{\nu_0}{\Longrightarrow}} \zeta$ for some
  $U$-valued random variable $\zeta$.
 
\item[\rm (ii)]
For any $\varepsilon>0$, it holds that
     $\mu \big[\rho(F_n \circ T, F_n)>\varepsilon\big] 
      \underset{n\to\infty}{\to} 0.$
\end{enumerate}
Then, it holds that 
    	$F_n \underset{n\to\infty}{\overset{\cL(\mu)}{\Longrightarrow}} \zeta.$
\end{Prop}

The stochastic process $(\eta^{(\alpha,\beta)}_j(t):t\geq0,\;j=1,\dots,d)$
is an $\bR^d$-valued $\alpha$-stable L\'evy process with  L\'evy measure 
\begin{align*}
	&\Pi_{(\alpha,\sum_{j=1}^d \beta_j \delta_{(\mathbbm{1}_{\{i=j\}})_{i=1}^d})}(A)
	\\
	&= \int_0^\infty \rd r \int_{\bS^{d-1}}
	 \sum_{j=1}^d\beta_j \delta_{(\mathbbm{1}_{\{i=j\}})_{i=1}^d}(\rd x)\mathbbm{1}_A(rx)\frac{\alpha r^{-\alpha-1}}{\Gamma(1-\alpha)},
	 \quad
	A\in\cB(\bR^d),
\end{align*}
where $\bS^{d-1}:=\{x\in\bR^d:|x|=1\}$ and $\delta_{x}$ denotes the Dirac measure at $x$.
We will modify the functional stable convergence for stationary sequence in Tyran-Kami\'nska \cite[Theorem 1.1]{Ty10b} and  obtain the following lemma.

\begin{Lem}[Functional convergence of partial sum of excursion lengths]\label{lem:fl-excursion}
Suppose that $T$ is a CEMPT on $(X,\cB,\mu)$ and that Assumptions \ref{ass:d-ray},  \ref{ass:reg-var} and \ref{ass:mixing} hold. 
Then, 
\begin{align}  
  \bigg(\frac{1}{n}\eta_j(b_nt):t\geq0,\;j=1,\dots,d\bigg)
  \underset{n\to\infty}{\overset{\cL(\mu)}{\Longrightarrow}}
  \big(\eta^{(\alpha,\beta)}_j(t):t\geq0, \; j=1,\dots,d\big),
  \notag\\
  \text{in $D([0,\infty),\bR^d)$}.
  \label{lem:fl-excursion-1}
\end{align}
\end{Lem}
\begin{proof}
By the assumptions and Proposition \ref{prop:stable},
we obtain
\begin{align}
F_n:=\bigg(\frac{1}{n}\eta_j(b_nt):t\geq0,\;j=1,\dots,d\bigg)
  \underset{n\to\infty}{\overset{\mu_Y}{\Longrightarrow}}
  \big(\eta^{(\alpha,\beta)}_j(t):t\geq0,\;i=1,\dots,d\big),
 \notag
 \\
 \text{in $D([0,\infty),\bR^d)$}.
  \label{lem:fl-excursion-2}
\end{align}
Let $\rho$ be the metric on $D([0,\infty),\bR^d)$ defined by \eqref{def:metric}.
By Proposition \ref{lem:st-dist}, it is sufficient to prove that, for any $\varepsilon>0$,
\begin{align}\label{lem:fl-excursion-goal}
	\mu[\rho(F_n\circ T, F_n)>\varepsilon]\underset{n\to\infty}{\to}0.
\end{align}
Let us prove \eqref{lem:fl-excursion-goal}. It is easily seen that 
\begin{align}
  \eta_j(t)\circ T 
   =
   \begin{cases}
   	 \eta_j(t)-1, & \text{on $T^{-1}A_j$},
   	 \\
   	 \eta_j(t), & \text{on $T^{-1}A_i\; (i\neq j)$},
   	 \\
   	 \eta_j(t+1), &  \text{on $T^{-1}Y$},
   	\end{cases}
\end{align}
For $t\geq0$ and $n\in\bN$, let us define the map $\lambda_n(t):X\to[0,\infty)$ by
\begin{align}
	\lambda_n(t)=\lambda_n(t)(x):=
	\begin{cases}
		t, &x\in T^{-1}Y^c,
		\\
		t+\min\{t/\sqrt{b_n}, 1\}, 
		&x\in T^{-1}Y.
	\end{cases}
\end{align}
Then we have
\begin{align}\label{inequality:time-change}
 	\max_{j=1,\dots,d}\sup_{t\geq0}\bigg|\frac{1}{n}\eta_j(b_n t)\circ T - \frac{1}{n}\eta_j(b_n\lambda_n(t))\bigg|
 	\leq 
 	\begin{cases}
 	1/n, & \text{on $T^{-1}Y^c$},
 	\\
 	\displaystyle
 	\max_{k\leq \sqrt{b_n}}\bigg(\frac{1}{n} \varphi \circ T^k_Y\bigg),
 	&\text{on $T^{-1}Y$}.
 	\end{cases}
\end{align}
For any $\varepsilon>0$, let $n$ be large enough so that $1/n<\varepsilon$ and $\log(1+1/\sqrt{b_n})<\varepsilon$.
By the definition of $\rho$ and the inequality \eqref{inequality:time-change}, we have
\begin{align}
 \mu[\rho(F_n\circ T, F_n)>\varepsilon]
 &\leq
 \mu\bigg[ T^{-1}Y \cap \bigg\{\max_{k\leq \sqrt{b_n}}
 \bigg(\frac{1}{n} \varphi \circ T^k_Y\bigg)>\varepsilon \bigg\}\bigg]
 \notag\\
 &=
 \mu\bigg[ Y \cap \bigg\{\max_{k\leq \sqrt{b_n}}
 \bigg(\frac{1}{n} \varphi \circ T^{k-1}_Y\bigg)>\varepsilon \bigg\}\bigg]
 \notag\\
 &\leq
 \mu(Y)\sqrt{b_n}\mu_Y[\varphi>n\varepsilon].
\end{align}
Here we used the fact that $T$ is $\mu$-preserving and $T_Y$ is $\mu_Y$-preserving. Since $\mu_Y[\varphi>n\varepsilon]=(\Gamma(1-\alpha) b_{\lfloor n\varepsilon +1 \rfloor})^{-1}$ and $b_n\in\cR_\alpha(\infty)$, we have
\begin{align}
	\sqrt{b_n}\mu_Y[\varphi>n\varepsilon]\underset{n\to\infty}{\to}0.
\end{align}
Therefore we obtain the desired convergence \eqref{lem:fl-excursion-goal}.
\end{proof}

We now prove Theorem \ref{main1}.

\begin{proof}[Proof of Theorem \ref{main1}]
Lemma \ref{lem:discrete-represent-G-D} implies that
\begin{align*}
 \frac{1}{n}G_Y(nt)
 &=
 \sup\bigg\{\frac{1}{n} \varphi(b_ns):\frac{1}{n} \varphi(b_ns)\leq t\bigg\},
 \quad t\geq0,
 \\
  \frac{1}{n}D_Y(nt)
 &=
 \inf\bigg\{\frac{1}{n} \varphi(b_ns):\frac{1}{n} \varphi(b_ns)> t\bigg\},
 \quad t\geq0.	
\end{align*}
For $a,b>0$ and $w\in D_0$, we will denote by $(aw(b\cdot))^{-1}$ the right-continuous inverse of $(aw(bt):t\geq0)$.
Then we see at once that $(aw(b\cdot))^{-1}(t)=\frac{1}{b}w^{-1}(\frac{1}{a}t)$, $t\geq0$. Therefore Lemma \ref{lem-dW} implies
\begin{align*}
	\frac{1}{n}S^{-1}_{A_j}(nt)
	=
	\frac{\lfloor nt+1\rfloor}{n}+ \sum_{\substack{i=1,\dots,d,\\ i\neq j}}
  \frac{1}{n}\eta_i\bigg(b_n
 \Big(\frac{1}{n}\eta_j(b_n\cdot)\Big)^{-1}(t)\bigg)
              +\frac{b_n}{n}\Big(\frac{1}{n}\eta_j(b_n \cdot)\Big)^{-1}(t),
  \\
  t\geq0,\;j=1,\dots,d. 
\end{align*}
and
\begin{align*}
	\frac{1}{n}\varphi(b_n t)&=\frac{\lfloor b_n t+1\rfloor}{n}+\sum_{i=1}^d \frac{1}{n}\eta_i(b_n t),
  \quad t\geq0.	
\end{align*}
Note that $b_n/n\to0$, as $n\to\infty$. 
Combining Lemmas \ref{lem:dist-conv-G-D} and \ref{lem:inverse} with Lemmas \ref{lem:represent-G-D}, \ref{lem:Wil} and \ref{lem:fl-excursion}, we have
\begin{align}
  &\bigg(\bigg(\frac{1}{n}S^{-1}_{A_j}(nt)\bigg)_{j=1}^d,\;
          \frac{1}{n}\varphi(b_nt), \;
          \frac{1}{n}G_Y(nt),\; 
          \frac{1}{n} D_Y(nt):t\geq0\bigg) 
    \notag\\
    &
    \underset{n\to\infty}{\overset{\cL(\mu)}{\Longrightarrow}}
    \Big(  \Big(\big(Z_j^{(\alpha,\beta)}\big)^{-1}(t)\Big)_{j=1}^d,
     \eta^{(\alpha,\beta)}(t),\; G^{(\alpha,\beta)}(t),\;
         D^{(\alpha,\beta)}(t) 
           :t\geq0\Big),
    \quad\text{in $D([0,\infty),\bR^{d+1})$}.	
\end{align}
Recall that $\varphi(t)=S_Y^{-1}(t)$ and $\eta^{(\alpha,\beta)}(t)=(L^{(\alpha,\beta)})^{-1}(t)$.
Using Corollary \ref{cor:inverse}, we obtain the desired convergence \eqref{main1-1}.	
\end{proof}

\begin{Rem}\label{rem:main-modif-proof}
We now return to the setting of Remark \ref{rem:sym-diff}. Combining \eqref{main1-1} with Birkhoff's ergodic theorem and Hopf's ratio ergodic theorem, we can easily have 
\begin{align}
    &\bigg(\bigg(\frac{1}{n}S_{A_j'}(nt)\bigg)_{j=1}^d,\;
          \frac{1}{b_n}S_{Y'}(nt):t\geq0\bigg) 
    \notag\\
    &
    \underset{n\to\infty}{\overset{\cL(\mu)}{\Longrightarrow}}
    \Big(  \big(Z_j^{(\alpha,\beta)}(t)\big)_{j=1}^d,\;
     \frac{\mu(Y')}{\mu(Y)}L^{(\alpha,\beta)}(t)
      :t\geq0\Big),
    \quad\text{in $D([0,\infty),\bR^{d+1})$}.
    \label{rem:main-modif-proof-1}	
\end{align}
Then we use Corollary \ref{cor:inverse} and Lemma \ref{lem:dist-conv-G-D} to obtain the convergence \eqref{rem:sym-diff-1}.	
\end{Rem}

\section{Proof of Theorem \ref{main2}}\label{sec:proof2}

We now return to the setting introduced in Subsection \ref{subsec:int} and prove Theorem \ref{main2} by using Theorem \ref{main1}. By Remark \ref{rem:sym-diff}, we only need to consider one particular combination of sets $A_1,\dots,A_d,Y\in\cB([0,1])$ satisfying (\ref{eq:AY}).
In the following, we will choose $A_1,\dots,A_d,Y$ suitably so that all of the conditions of Assumptions \ref{ass:d-ray}, \ref{ass:reg-var} and \ref{ass:mixing} will  be satisfied. 
We need to break up the proof into the case of $d=2$ and the case of $d\geq3$ for a certain reason. See Remark \ref{rem:main2}.

\subsection{Case $d=2$}

  Let us consider the case of $d=2$. Following Thaler \cite[Section 4]{T02} and Zweim\"uller \cite[Section 2]{Z03}, we choose a point $\gamma\in J_1$ such that
\begin{align}
	T\gamma\in J_2
	\quad\text{and}\quad
	T^2\gamma=\gamma.
\end{align}
Hence $\gamma$ is a $2$-periodic point of $T$. See Figure \ref{figure:periodic}. Set
\begin{align}
  A_1:=[0,\gamma),
  \quad
  Y:=[\gamma,T\gamma]
  \quad\text{and}\quad
  A_2:=(T\gamma,1].	
\end{align}
Then we have
\begin{align*}
	&
	T(A_1) = A_2^c,
    \quad
	T(A_2) = A_1^c,
	\quad
	T(Y\cap J_1) = A_2,
	\quad
	T(Y\cap J_2) = A_1, 
	\quad \text{a.e.}
\end{align*}
Therefore Assumption \ref{ass:d-ray} holds for $d=2$. 
Let us define $\varphi$, $\ell=(\ell_1,\ell_2)$, $\mu_Y$ and $T_Y$ as in Section \ref{sec:notation}.
For $n\geq1$, we define subsets $P_{1,n}$ and $P_{2,n}\subset Y$ by 
\begin{align}
	P_{1,n}
	:=
	Y\cap \{\ell_1=n\}
	=
	Y\cap J_2\cap\{\varphi=n+1\},
	\\
	P_{2,n}
	:=
	Y \cap \{\ell_2=n\}
	=
	Y \cap J_1\cap \{\varphi=n+1\}.
\end{align}
Then we have 
\begin{align}
  P_{1,n}= (f_2\circ f^n_1)(Y) =\Big[(f_2\circ f^{n}_1)\gamma, \:(f_2\circ f_1^{n-1})\gamma\Big),
  \\
  P_{2,n} = (f_1\circ f^n_2)(Y) =\Big((f_1\circ f^{n}_2)\gamma, \:(f_1\circ f_2^{n+1})\gamma\Big].	
\end{align}
Here we used the fact that $T\gamma = f_1^{-1}\gamma = f_2 \gamma$.
It is easy to check that $Y=\sum_{j,n}P_{j,n}$, a.e.

\begin{figure}
\begin{minipage}{0.5\hsize}
\centering
\includegraphics[width=0.5\linewidth]{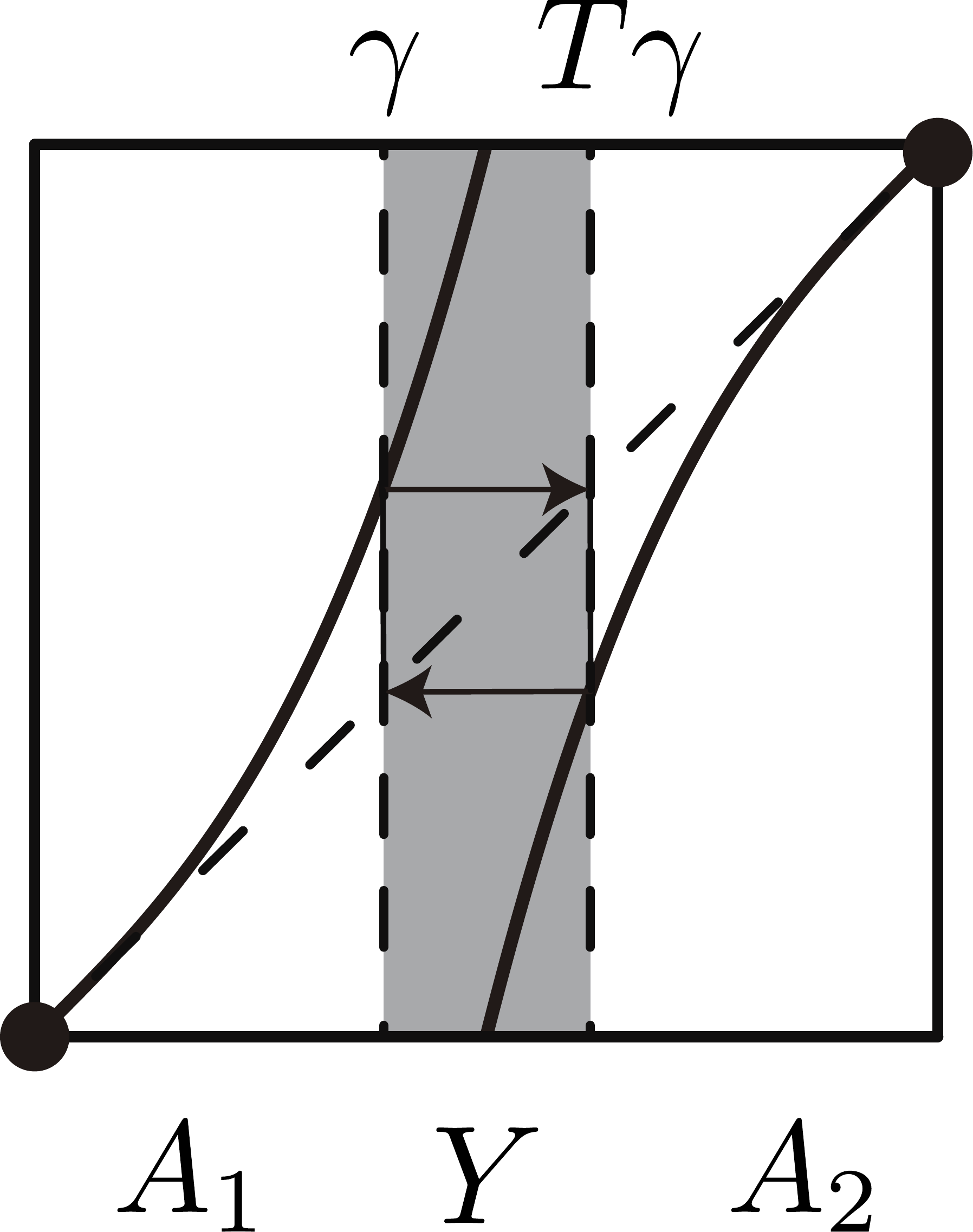}
\caption{2-periodic point $\gamma$}\label{figure:periodic}
\end{minipage}
\begin{minipage}{0.5\hsize}
\centering
\includegraphics[width=0.6\linewidth]{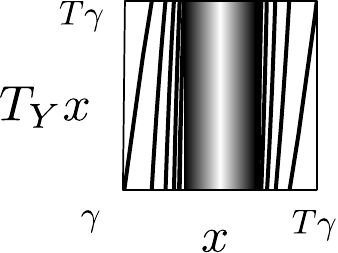}
\caption{First return map $T_Y:Y \to Y$, which has infinitely many branches $T_Y|_{P_{j,n}}: P_{j,n}\to Y$ ($j=1,2$; $n\geq0$)}
\end{minipage}
\end{figure}

\begin{Lem}[Thaler {\cite[Lemma 5]{T02}}]\label{lem:interval-2-reg}
 Suppose that Assumption \ref{ass:indiff} holds. Then,  
 \begin{align}
 	\mu_Y[\ell_j\geq n]
 	\underset{n\to\infty}{\sim}
 	\beta_j \frac{1}{\mu(Y)\Gamma(1-\alpha)b_n},
 	\quad
 	j=1,2.
 \end{align}
where $\beta_j$ and $(b_n)_{n\geq0}$ have been defined by {\rm (\ref{beta_j})} and {\rm (\ref{phi})}, respectively.	
\end{Lem}

Combining Lemma \ref{lem:interval-2-reg} with Karamata's Tauberian theorem and the equalities \eqref{eq:difference-1} and \eqref{eq:difference-2}, we see that Assumption \ref{ass:reg-var} is satisfied.

\begin{Lem}[Zweim\"uller {\cite[Lemma 2]{Z03}}]\label{lem:interval-2-exp-mixing}
	The map $T_Y:Y \to Y$ satisfies the following conditions:
\begin{itemize}
\item[\rm (1)] For each $j=1,2$ and $n\geq1$, the restriction $T_Y|_{P_{j,n}}$ can be extended to a $C^2$-bijective map  from $\bar{P_{j,n}}$ to $Y$.
\item[\rm (2)] 
$\inf\{T_Y'x:x\in\sum_{j,n} P_{j,n}\}>1$.
\item[\rm (3)] 
$\sup\{|T_Y''x|/|T_Y'x|^2:x\in\sum_{j,n} P_{j,n}\}<\infty$.
\end{itemize}
\end{Lem}

Combining Lemmas \ref{lem:interval-2-reg} and \ref{lem:interval-2-exp-mixing} with Proposition \ref{lem:Markov-exp-mixing}, we see that Assumption \ref{ass:mixing} is satisfied. 
Therefore we use Theorem \ref{main1} to obtain the desired result.

\subsection{Case $d\geq3$}

Let us consider the case of $d\geq3$. Set
\begin{align}
 A_j
 := 
 J_j\cap T(J_j),
 \quad j=1,\dots,d,
 \label{Aj-case3}
 \quad\text{and}\quad
 Y
 :=
 [0,1] \setminus \sum_{j=1}^d A_j.	
\end{align}
See Figure \ref{fig:case3}. Then Assumption \ref{ass:d-ray} holds. Let us define $\varphi$, $\ell=(\ell_1,\dots,\ell_d)$, $\mu_Y$ and $T_Y$ as in Section \ref{sec:notation}. The following lemma is a slight modification of Lemma \ref{lem:interval-2-reg}. 
\begin{Lem}
 Suppose that Assumption \ref{ass:indiff} holds. Then,  
 \begin{align}
 	\mu_Y[\ell_j\geq n]
 	\underset{n\to\infty}{\sim}
 	\beta_j \frac{1}{\mu(Y)\Gamma(1-\alpha)b_n},
 	\quad
 	j=1,\dots,d,
 \end{align}
where $\beta_j$ and $(b_n)_{n\geq0}$ have been defined by {\rm (\ref{beta_j})} and {\rm (\ref{phi})}, respectively.	
\end{Lem}

\begin{figure}
\begin{minipage}{0.5\hsize}
\centering
\includegraphics[width=0.6\linewidth]{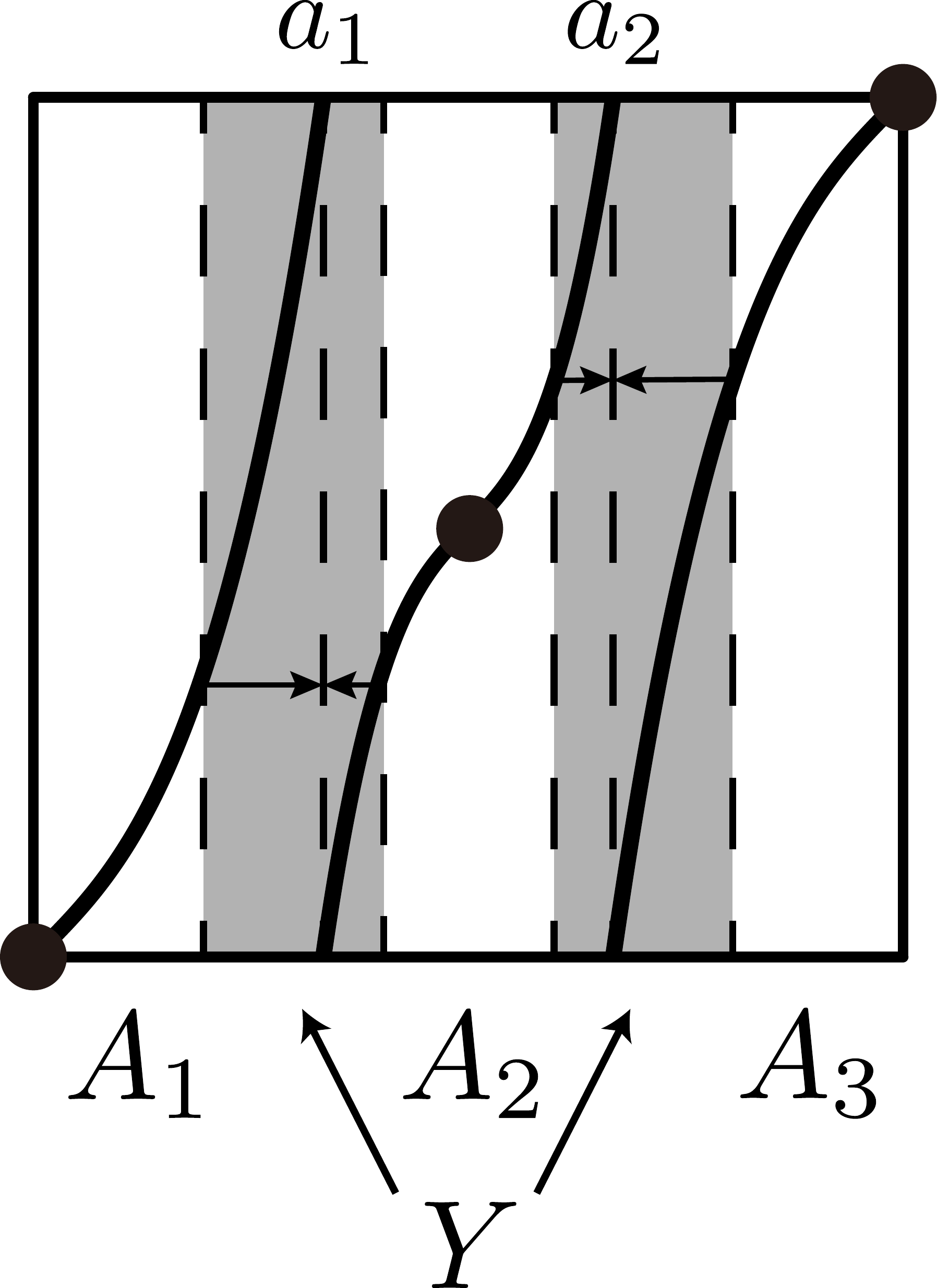}
\caption{Case $d=3$}
\label{fig:case3}	
\end{minipage}
\begin{minipage}{0.5\hsize}
\centering
\includegraphics[width=0.6\linewidth]{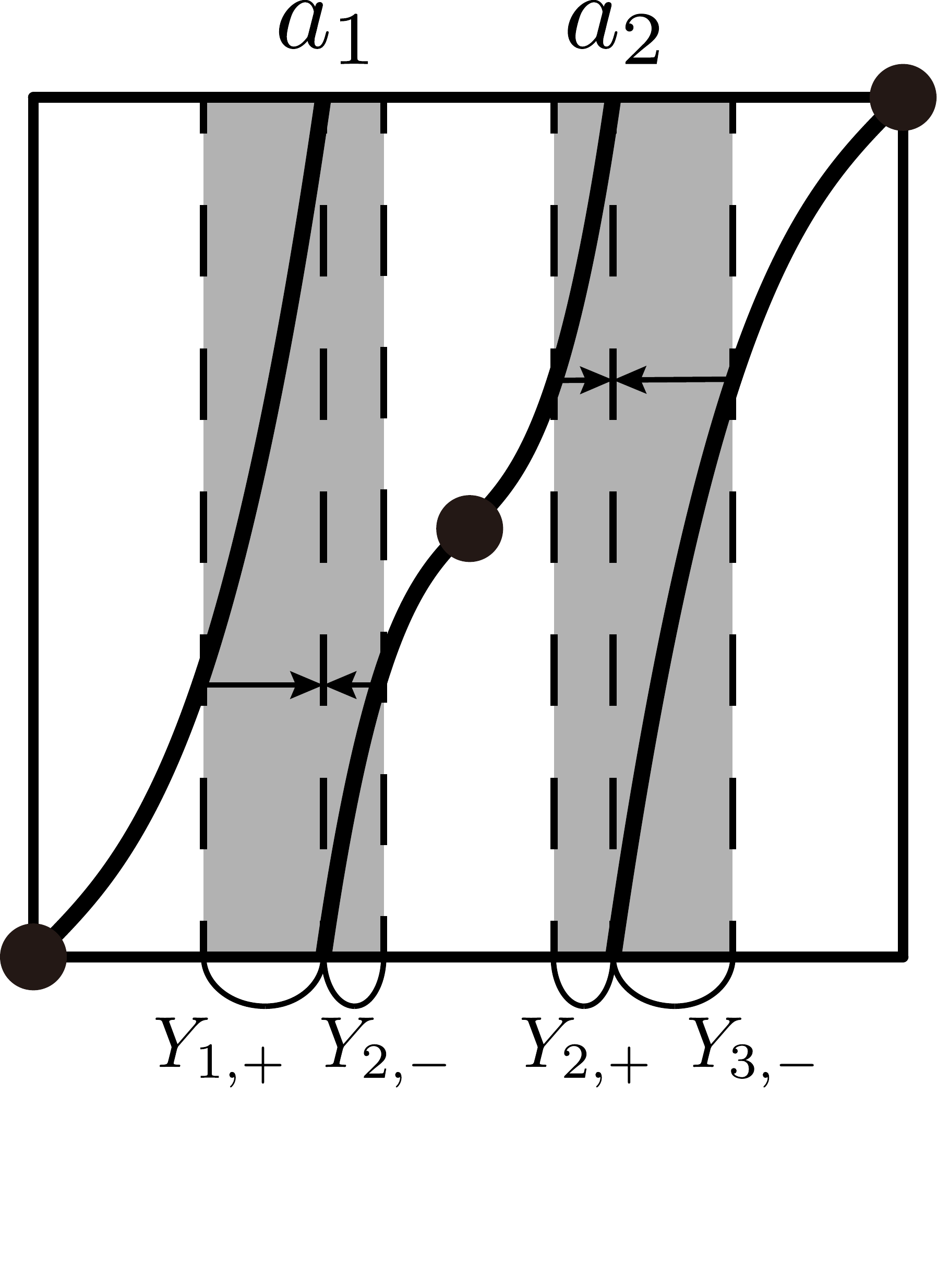}
\caption{$Y_{j,\pm}$}
\label{fig:case3-1}		
\end{minipage}
\end{figure}

Set
\begin{align}
	Y_{j,-}:= Y\cap (a_{j-1},x_j) 
	\quad
	\text{and}
	\quad
	Y_{j,+}:= Y\cap  (x_j,a_j),
	\quad
	j=1,\dots,d.
	\label{J-case3}
\end{align}
See Figure \ref{fig:case3-1}.
For $i,j \in\{1,\dots,d\}$, $\sigma\in\{+,-\}$ and $n\geq0$, we define a subset $P_{i,j,\sigma,n}\subset Y$ by
 \begin{align}
 P_{i,j,\sigma,n}
 :=&\:
 Y\cap J_i\cap T_Y^{-1}(Y_{j,\sigma}) \cap\{\ell_j=n\}
 \notag
 \\
 =&\:
 Y\cap J_i\cap T_Y^{-1}(Y_{j,\sigma}) \cap\{\varphi=n+1\}.
 \label{P-case3}
 \end{align}
We see at once that
\begin{align}\label{partition}
	Y_{i_0,\sigma_0}
  =
  \begin{cases}
  \sum_{j<i_0,\;\sigma=\pm,\;n\geq0} P_{i_0,j,\sigma,n},
  &
  \text{if $\sigma_0=-$},
  \\
  \sum_{j>i_0,\;\sigma=\pm,\;n\geq0} P_{i_0,j,\sigma,n}
  &
  \text{if $\sigma_0=+$},	
  \end{cases}
  \quad\text{a.e.},
\end{align}
and hence $Y=\sum_{i,j,\sigma,n}P_{i,j,\sigma,n}$, a.e.
Set 
\begin{align*}
  \Theta :=\{(i,j,\sigma,n):P_{i,j,\sigma,n}\neq \emptyset \}
  	     =\{(i,j,\sigma,n):i\neq j \;\;\text{and}\;\; (j,\sigma)\neq(1,-) , (d,+)\}.
\end{align*}

\begin{Lem}
	The map $T_Y:Y\to Y$ satisfies the following conditions:
\begin{itemize}
\item[\rm (1)] For each $(i,j,\sigma,n)\in \Theta$, the restriction $T_Y|_{P_{i,j,\sigma,n}}$ can be extended to a $C^2$-bijective map from $\bar{P_{i,j,\sigma,n}}$ to $\bar{Y_{j,\sigma}}$.
\item[\rm (2)] 
$\inf\{T_Y'x:x\in\sum_{i,j,\sigma,n} P_{i,j,\sigma,n}\}>1$.
\item[\rm (3)] 
$\sup\{|T_Y''x|/|T_Y'x|^2:x\in\sum_{i,j,\sigma,n} P_{i,j,\sigma,n}\}<\infty$.
\item[\rm (4)] For each $(i,j,\sigma,n)\in \Theta$, it holds that $T_Y^4(P_{i,j,\sigma,n})=Y$, a.e.
\end{itemize}
\end{Lem}

\begin{proof}
The proofs of (1), (2) and (3) are almost the same as those of Lemma \ref{lem:interval-2-exp-mixing}. So we omit them. 
Let $(i,j,-,n)\in \Theta$. Then, we have $T_Y(P_{i,j,-,n})= Y_{j,-}$, a.e.
Using (\ref{partition}), we have $T_Y(Y_{j,-})\supset Y_{1,+}$, a.e.
We see at once that
\begin{align*}
  T_Y(J_{1,+})= Y\setminus Y_{1,+}
  \quad
  \text{and}
  \quad
  T_Y(Y\setminus Y_{1,+}) = Y,
  \quad	\text{a.e.}
\end{align*}
Therefore we obtain $T_Y^4(P_{i,j,-,n})=Y$, a.e. Similarly, we can obtain $T^4_Y(P_{i',j',+,n'})=Y$, a.e., for $(i',j',+,n')\in\Theta$.	
\end{proof}

Therefore we obtain the desired result in the case of $d\geq3$,  as in the case of $d=2$.

\begin{Rem}\label{rem:main2}
Let us consider the case of $d=2$ and define $A_j$, $Y$, $Y_{i,\sigma}$ and $P_{i,j,\sigma,n}$ by (\ref{Aj-case3}), (\ref{J-case3}) and (\ref{P-case3}), respectively.
Then we have 
\begin{align*}
   T_Y(Y_{1,+})=Y_{2,-} 
   \quad
   \text{and}
   \quad
   T_Y(Y_{2,-})=Y_{1,+},
   \quad
   \text{a.e.}
\end{align*}
Hence 
$T_Y^{m}(P_{i,j,\sigma,n}) \neq Y$, a.e., for any $m\geq1$, that is, $T_Y$ is not aperiodic. Therefore the choice \eqref{Aj-case3} is not suitable in the case of $d=2$.	
\end{Rem}

\appendix

\section{Mixing property of uniformly expanding Markov  interval map}\label{sec:markov-map}
We will recall mixing properties of uniformly expanding Markov interval maps. For  basic discussions of Markov interval maps, see for instance Adler \cite{Adler}, Bowen \cite{Bo}, Bowen--Series \cite{BS} and Pollicott--Yuri \cite[Sections 4 and 12]{PY}.

Let $(P_i)_{i\geq 1}$ be a countable family of disjoint open subintervals of $(0,1)$, and let $Y=\sum_{i\geq1}P_i$, a.e.
Suppose that a map $F:Y\to Y$ satisfies the following conditions: 	
\begin{itemize}
\item[(1)] ($C^2$-extension) for each $i\geq1$, the restriction $F|_{P_i}$ can be extended to a $C^2$-function on $\bar{P_i}$.
\item[(2)] (Markov map) If $F(P_i)\cap P_j\neq \emptyset$ for some $i,j\geq1$, then $F(P_i)\supset P_j$.

\item[(3)] (aperiodicity) There exists $n_0\geq1$ such that, for any $i\geq1$, it holds that $F^{n_0} (P_i)=Y$, a.e.  

\item[(4)] (finite image)  $\{F(P_i):i\geq1\}$ is a finite collection.

\item[(5)] (uniformly expanding) 
$\inf\{|F'(x)|:x\in\sum_i P_i\}>1$.

\item[(6)] (R\'enyi and Adler's condition)
$\sup\{|F''(x)|/|F'(x)|^2:x\in\sum_i P_i\}<\infty$.

\end{itemize}

The following two propositions are slight modifications of \cite[Theorem (I.2)]{BS} and \cite[Theorem 12.5]{PY}, respectively. So we omit the proofs of them.

\begin{Prop}\label{lem:invariant}
Assume that the conditions {\rm (1)--(6)} hold. Then
the map $F$ has a unique invariant probability measure $\nu_0$ equivalent to the Lebesgue measure on $Y$.  
\end{Prop}

\begin{Prop}\label{lem:exact}
	Assume that the conditions {\rm (1)--(6)} hold. Let $\nu_0$ be the $F$-invariant probability measure given by Proposition \ref{lem:invariant}.
	Then the map $F$ is exact (and hence strong mixing) with respect to $\nu_0$, that is, 
$\bigcap_{n=0}^\infty \{F^{-n} A:A\in\cB(Y)\}=\{\emptyset, Y\}$, $\nu_0$-a.e.
\end{Prop}

Let us define a sub-$\sigma$-field $\cF_n^m\subset \cB(Y)$ by $\cF_n^m:=\sigma\{F^{-k}P_i:\text{$n\leq k\leq m$ and $i\geq1$}\}$.
Combining Propositions \ref{lem:invariant} and \ref{lem:exact} with \cite[Theorem 1.(b)]{AN} or \cite[Corollary 4.7.8]{A97}, we obtain the following proposition.

\begin{Prop}\label{lem:Markov-exp-mixing}
Assume that the conditions {\rm (1)--(6)} hold. Let $\nu_0$ be the $F$-invariant probability measure given by Proposition \ref{lem:invariant}. 
Then there exist $C\in(0,\infty)$ and $\theta\in(0,1)$ such that, for any $k,n\geq0$, $A\in\cF_0^k$ and $B\in\cB(Y)$, 
\begin{align}
	\big|\nu_0(A\cap F^{-(k+n)}(B))-\nu_0(A)\nu_0(B)\big|\leq C \theta^n \nu_0(A)\nu_0(B).
\end{align}
\end{Prop}

\section{Functional convergence to $\alpha$-stable L\'evy process}\label{sec:stable}

Following Tyran-Kami\'nska \cite{Ty10b}, we will explain a functional limit theorem for the processes which have stationary increments and local dependence. 

Let $d\geq1$ be a positive integer, and $(Z_n)_{n\geq1}$ a strictly stationary sequence of $\bR^d$-valued random variables. 
Set $\bS^{d-1}:=\{x\in \bR^d:|x|=1\}$, where $|\cdot|$ denotes the Euclid norm on $\bR^d$. We will denote by $\cP_{\bS^{d-1}}$ the class of probability measures on $\bS^{d-1}$. We endow $\cP_{\bS^{d-1}}$ with the Polish topology of weak convergence.

\begin{Ass}[Regular variation]\label{ass:stationary-reg}
The random variable $Z_1$ is \emph{regular varying} with index $\alpha\in(0,1)$ and spectral measure $\rho \in \cP_{\bS^{d-1}}$, that is, it holds that
\begin{align*}
  &\bP\Big[|Z_1|>\lambda r\;\Big|\;|Z_1|>r \Big]
  \underset{r\to\infty}{\to}
  \lambda^{-\alpha}, \;\;\text{in $\bR$, for $\lambda>0$},	
  \\ 
  &\text{and}\quad
  \bP\bigg[\frac{Z_1}{|Z_1|}\in \cdot\;\bigg|\;|Z_1|>r \bigg]
  \underset{r\to\infty}{\to}
  \rho(\cdot),
  \;\;\text{in $\cP_{\bS^{d-1}}$},	
\end{align*}
where $\bP[\cdot|\cdot]$ denotes the conditional probability.
\end{Ass}

We will assume Assumption \ref{ass:stationary-reg} from now on. For $n\geq1$, set
\begin{align}
  b_n:=\frac{1}{\Gamma(1-\alpha)\bP[|Z_1|>n]}\in(0,\infty).	
\end{align}
We define a sub-$\sigma$-field $\cF_n^m\subset\cB(\bR^d)$ by $\cF_n^m:=\sigma\{Z_k:n\leq k \leq m\}$. For $n\geq1$, set
\begin{align}
  \phi_0(n)
  :=
  \sup
 \{|\bP(A \cap B)-\bP(A)\bP(B)|
   :
k\geq0,\;A\in\cF_{0}^k,\;
B\in\cF_{k+n}^\infty
\}.      	
\end{align}	
\begin{Ass}[Local dependence]\label{ass:stationary-mixing}
	For any $\varepsilon>0$, there exist $\bN$-valued sequences $(r_n)_{n\geq1}$ and $(s_n)_{n\geq1}$ such that
\begin{align*}
	s_n, \;\frac{r_n}{s_n},\; \frac{b_n}{r_n}\underset{n\to\infty}{\to}\infty
	\quad
	\text{and}
	\quad\frac{b_n}{r_n}\phi_0(s_n)\underset{n\to\infty}{\to}0,
\end{align*}
and
\begin{align*}
 \bP\Big[ \max_{2\leq k \leq r_n} |Z_k|>\varepsilon n 
       \;\Big|\; |Z_1|>\varepsilon n\Big]
       \underset{n\to\infty}{\to}0.	
\end{align*}
\end{Ass}

\begin{Rem}
Suppose that Assumption \ref{ass:stationary-reg} is satisfied. Furthermore, assume that $(Z_n)_{n\geq1}$ is exponentially continued fraction mixing, that is, 
there exist some constants $C\in(0,\infty)$ and $\theta\in(0,1)$ such that, for any $k,n\geq0$, $A\in\cF_0^k$ and $B\in\cF_{k+n}^\infty$,
\begin{align}
 |\bP(A\cap B)-\bP(A)\bP(B)|
 \leq
 C\theta^n
 \bP(A)\bP(B).	
\end{align}
Then Assumption \ref{ass:stationary-mixing} is satisfied. See Tyran-Kami\'nska \cite[Subsection 4.1]{Ty10a} for the details.
\end{Rem}

For $n\geq1$, we define an $\bR^d$-valued c\`adl\`ag process $X_n=(X_n(t):t\geq0)$ by
\begin{align}
  X_n(t):=\frac{1}{n}\sum_{k=1}^{\lfloor b_nt\rfloor}Z_k, 
  \quad t\geq0.
\end{align}
We will denote by $X_{(\alpha,\rho)}=(X_{(\alpha,\rho)}(t):t\geq0)$ an $\bR^d$-valued $\alpha$-stable L\'evy process with  L\'evy measure $\Pi_{(\alpha,\rho)}$ given by 
\begin{align}\label{Levy-measure}
	\Pi_{(\alpha,\rho)}(A)
	:=
	\int_0^\infty\rd r
	\int_{\bS^{d-1}}\rho(\rd x)
	 \mathbbm{1}_A(rx)\frac{\alpha r^{-\alpha-1}}{\Gamma(1-\alpha)},
	\quad
	A\in\cB(\bR^d). 
\end{align}
The following proposition is a slight modification of Tyran-Kami\'nska \cite[Theorem 1.1]{Ty10b}. So we omit its proof. 
\begin{Prop}[Functional convergence to stable process]\label{prop:stable}
Suppose that Assumptions \ref{ass:stationary-reg} and \ref{ass:stationary-mixing} hold. Then, it holds that 
 \begin{align}
 	X_n
 	\underset{n\to\infty}{\dto}
 	X_{(\alpha,\rho)},
 	\quad
 	\text{in $D([0,\infty),\bR^d)$.}
 \end{align}
\end{Prop}

\end{document}